\documentclass[10pt]{amsart}

\usepackage{amsthm, amscd}
\usepackage{amssymb}
\usepackage[all]{xy}
\usepackage{hyperref}

\usepackage{ifpdf}
\ifpdf
\xyoption{pdf}
\else
\fi

\numberwithin{equation}{section}
\numberwithin{figure}{section}
\theoremstyle{plain}

\newtheorem{thm}{Theorem}[section]
\newtheorem{lem}[thm]{Lemma}
\newtheorem{defn}[thm]{Definition}
\newtheorem{prop}[thm]{Proposition}

\newtheorem{rmk}[thm]{Remark}

\begin{document}

\title{Distinguishing between exotic symplectic structures}
\author{Richard M. Harris}
\address{DPMMS, Centre for Mathematical Sciences, Wilberforce Road, Cambridge,
CB3 0WB.}
\email{rrricharrrd@gmail.com}

\begin{abstract}
We investigate the uniqueness of so-called exotic structures on certain exact symplectic manifolds by looking at how their symplectic properties change under small nonexact deformations of the symplectic form.  This allows us to distinguish between two examples based on those found in \cite{maydanskiy,maydanskiyseidel}, even though their classical symplectic invariants such as symplectic cohomology vanish.  We also exhibit, for any $n$, an exact symplectic manifold with $n$ distinct but exotic symplectic structures, which again cannot be distinguished by symplectic cohomology.
\end{abstract}

\maketitle

\renewcommand{\AA}{\mathcal{A}}
\newcommand{\BB}{\mathcal{B}}
\newcommand{\DD}{\mathbb{D}}
\newcommand{\FF}{\mathcal{F}}
\newcommand{\JJ}{\mathcal{J}}
\newcommand{\jj}{\mathbb{J}}
\newcommand{\PP}{\mathcal{P}}
\newcommand{\LL}{\mathcal{L}}
\newcommand{\WW}{\mathcal{W}}
\newcommand{\UU}{\mathcal{U}}
\newcommand{\MM}{\mathcal{M}}
\newcommand{\CC}{\mathbb{C}}
\newcommand{\RR}{\mathbb{R}}
\newcommand{\ZZ}{\mathbb{Z}}
\newcommand{\bj}{\mathbf{J}}
\newcommand{\dbar}{\bar{\partial}_{\bj}}
\renewcommand{\o}{\omega}
\newcommand{\mo}{(M,\o)}
\newcommand{\ainf}{A_{\infty}}
\newcommand{\vdim}{\operatorname{vdim}}
\newcommand{\supp}{\operatorname{supp}}
\newcommand{\ppk}{\PP^p_{k}}
\newcommand{\lpk}{L^p_{k;loc}}
\newcommand{\wpk}{W^p_{k;loc}}
\newcommand{\rk}{\operatorname{rank}}
\newcommand{\im}{\operatorname{im}}
\newcommand{\Id}{\operatorname{Id}}
\newcommand{\End}{\operatorname{End}}

\section{Introduction}
This paper concerns the uniqueness of exact symplectic structures on Liouville
domains, an area which has seen considerable recent development.  In many
situations, such as those coming from cotangent bundles or affine varieties, a
Liouville domain $M$ carries what is considered to be a ``standard'' symplectic
form.  As we shall recap in this introduction, there are now known to be many
examples of Liouville domains with exact symplectic forms which are not
equivalent to the standard ones.  Any such form will be called ``exotic'' in
this paper.

Historically, Gromov \cite{gromov} was the first to exhibit a nonstandard exact
symplectic structure on Euclidean space, although, whereas the standard
symplectic structure is Liouville, Gromov's is not known to be (see Section
\ref{sec:lef} for the relevant definitions).  The first exotic structures on
$\RR^{4n}$ (for $4n \geq 8$) known to be Liouville were discovered by
Seidel-Smith \cite{seidelsmithramanujam}, later extended by McLean
\cite{mcleanlefschetz} to cover all even dimensions greater than 8.  McLean
actually found infinitely many such pairwise-distinct nonstandard symplectic
structures, which were all distinguished by considering their symplectic
cohomology $SH^*(M)$.

More recently, Fukaya categorical techniques have been used by
Maydanskiy-Seidel \cite{maydanskiyseidel} (refining earlier work of Maydanskiy
\cite{maydanskiy}) to find exotic symplectic structures on $T^*S^n$ (for $n
\geq 3$).  These are shown to be nonstandard by proving that they contain no
homologically essential exact Lagrangian $S^n$, in contrast to the zero-section
for the standard symplectic form.  Similar results have also been obtained
using the work of Bourgeois-Ekholm-Eliashberg \cite{bee} again using
symplectic/contact cohomology-type invariants.  Such results have been further
extended by Abouzaid-Seidel \cite{abouzaidseidelrecombination} to show the
existence of infinitely many distinct exotic structures on any affine variety
of real dimension $\geq 6$, again distinguished using symplectic cohomology.

In this paper, we shall consider six-dimensional symplectic manifolds of the
types considered by Maydanskiy \cite{maydanskiy} and Maydanskiy-Seidel
\cite{maydanskiyseidel}.  In \cite{maydanskiyseidel}, infinitely many ways are
presented of constructing a nonstandard $T^*S^3$, but the question of whether
all these constructions actually yield symplectically distinct manifolds is
left open.  We shall not answer that question, but instead we shall consider
what happens when we add a 2-handle to such an exotic $T^*S^3$.  The result
will be diffeomorphic to a manifold constructed in \cite{maydanskiy}, which
again contains no exact Lagrangian $S^3$.

Specifically, we shall consider the manifolds given by the diagrams in Figure
\ref{fig:main}. The meaning of such diagrams will be explained in Section
\ref{sec:lef}. Briefly, our main method of constructing symplectic manifolds
$E^6$ will be as Lefschetz fibrations over $\CC$.  To run this construction,
the input data consists of a symplectic manifold $M^4$ and an ordered
collection of Lagrangian spheres in $M^4$ (see Lemma \ref{thm:build}).  In
Figure \ref{fig:main}, we can associate to each path some Lagrangian sphere in
a 4-dimensional Milnor fibre, which is our required data.

\begin{figure}[h!]
\[
\xy
(-30,0)*{\xy
(-10,0)*{\times}; (10,0)*{\times}; **\crv{(0,13)};
(-10,0)*{}; (10,0)*{}; **\crv{(0,-13)};
(0,0)*{\times};
\endxy};
(30,0)*{\xy
(-20,0)*{\times}; (10,0)*{\times}; **\dir{-};
(0,0)*{\times};
(-10,0)*{\times};
(20,0)*{\times};
(-20,0); (10,0); **\crv{(-10,-18)&(1,20)};
\endxy};
(-30,-10)*{X_1};
(30,-10)*{X_2};
\endxy
\]
\caption{}
\label{fig:main}
\end{figure}

Each of these spaces is diffeomorphic to $T^*S^3 \cup 2$-handle.  There is a
standard way of attaching a 2-handle to $T^*S^3$ \cite{weinsteinhandle} such
that we still get an exact Lagrangian sphere inherited from the zero-section.
However, neither $X_1$ nor $X_2$ contains such a sphere, so are considered
exotic.  In addition, $X_1,X_2$ both have vanishing symplectic cohomology.
This is not proved in \cite{maydanskiy,maydanskiyseidel} and so we include this
calculation in Section \ref{sec:sh}, and has the consequence (already proven
for $X_1$ in \cite{maydanskiy}) that $X_1$ and $X_2$ actually contain no exact
Lagrangian submanifolds (such symplectic manifolds are sometimes called
``empty'').  Despite the usual collection of invariants being insufficient to
distinguish these two manifolds, we shall nevertheless prove

\begin{thm}
\label{thm:main}
$X_1$ and $X_2$ are not symplectomorphic.
\end{thm}

We shall then extend our methods to prove

\begin{thm}
\label{thm:many}
Pick any $n \geq 1$.  Then there exists a manifold $M$ (diffeomorphic to
$T^*S^3$ with $n$ 2-handles attached), which supports exact symplectic forms
$\o_1, \ldots, \o_{n+1}$ such that, with respect to each $\o_i$, $(M, \o_i)$ is
Liouville and contains no exact Lagrangian submanifolds, but such that there
exists no diffeomorphism $\phi$ of $M$ such that $\phi^*\o_i = \o_j$ for $i
\neq j$.
\end{thm}

The main technique used in this paper is to consider what happens after a
nonexact deformation of the symplectic structure.  For any 2-form $\beta \in
H^2(X_i ; \RR)$, we can consider an arbitrarily small nonexact deformation of
$\o$ to $\o + \epsilon \beta$.  If this new form is still symplectic, we can
look at the symplectic properties of these new symplectic manifolds.  (In the
case of $X_1$ and $X_2$ above $H^2(X_i;\RR)=\RR$, so Moser's argument tells us
that the way we can perform such a deformation is essentially unique, in a
sense which will be made precise in Section \ref{sec:end}.)  We discover that,
after an arbitrarily small deformation, $X_1$ (which with our original exact
form contains no Lagrangian $S^3$) does in fact contain such a sphere, an
interesting phenomenon in its own right which is explained in Section
\ref{sec:maydanskiy}.

In contrast, after such a deformation, $X_2$ still contains no homologically
essential Lagrangian sphere.  The proof of this fact requires rerunning the
argument of \cite{maydanskiyseidel}, except that somewhat more care needs to be
exercised in the use of Floer cohomology groups, owing to the nonexactness of
our deformed situation.  This is the content of Section
\ref{sec:maydanskiyseidel}.

In general, given a symplectic manifold $M$ (satisfying certain topological
assumptions), we can consider the set $\Gamma_1 \subset
\mathbb{P}(H^2(M;\RR))$, of directions in which we get no homologically
essential Lagrangian sphere inside $M$ after an arbitrarily small deformation
of the symplectic form in that direction.  We show that this is a symplectic
invariant, which completes the proof of Theorem \ref{thm:main}.  Finally, in
Section \ref{sec:many} these ideas are extended to prove Theorem
\ref{thm:many}.

\textsl{Acknowledgements} I would like to thank Paul Seidel whose idea it
originally was to distinguish such exotic manifolds by deforming their
symplectic structures.  I would also like to thank my PhD supervisor Ivan Smith
for many helpful discussions, as well as for useful comments on earlier drafts
of this work.  Thanks also go to Gabriel Paternain, Richard Thomas and the anonymous referee for pointing out several errors in earlier drafts of this paper.  During this research I was partially
supported by European Research Council grant ERC-2007-StG-205349

\section{Lefschetz fibrations}
\label{sec:lef}
In this section, we recall the standard notions of Picard-Lefschetz theory. The
treatment here largely follows that of \cite[Part III]{seidelbible}, but we
shall adapt the presentation there to include certain nonexact symplectic
manifolds, as we want to consider arguments involving nonexact deformations of
our symplectic form.

Let $\mo$ be a noncompact symplectic manifold.  We say $\mo$ is \emph{convex at
infinity} if there exists a contact manifold $(Y,\alpha)$ which splits $M$ into
two parts: a relatively compact set $M^{in}$; and $M^{out}$, which is
diffeomorphic to the positive
symplectization of $(Y,\alpha)$ where, in a neighbourhood of $Y$, we have a
1-form $\theta$ satisfying $d\theta = \o$ and $\theta|_Y=\alpha$.  Such a
contact manifold is canonically identified up to contactomorphism.  If $\theta$
can be defined on the whole of $M$, we call $(M,\theta)$ a \emph{Liouville
manifold}.

Given a compact symplectic manifold with boundary $M$ such that, in a
neighbourhood
of the boundary, we have a primitive $\theta$ of the symplectic form which
makes
the boundary contact, we say $M$ has \emph{convex boundary}.  If $\theta$ is
defined everywhere, $(M,\theta)$ is usually called a \emph{Liouville domain}.
Given a symplectic manifold with convex boundary, we can complete $M$
canonically to get a symplectic manifold convex at infinity,
\[
\widehat{M}=M\cup_{\partial M} [0,\infty) \times \partial M,
\]
with forms $\widehat{\theta} = e^r \theta$ and $\widehat{\o} = d
\widehat{\theta}$ on the collar, where $r$ denotes the coordinate on
$[0,\infty)$.

\subsection{Definition}
Let $(E,\o)$ be a compact symplectic manifold with corners such that, near the
boundary, $\o=d\theta$ for some form $\theta$ which makes the codimension 1
strata contact, and let $\pi \colon E \to S$ be a proper map to a compact
Riemann surface with boundary such that the following conditions hold:

\begin{itemize}
 \item There exists a finite set $E^{crit} \subset E$ such that $D\pi_x$ is a
submersion for all $x \notin E^{crit}$, and such that $D^2\pi_x$ is
nondegenerate for all $x \in E^{crit}$, which means that locally we can find
charts such that $\pi(z) = \sum{z_i^2}$ .  We denote by $S^{crit}$ the image of
$E^{crit}$ and require that $S^{crit} \subset S \setminus \partial S$.  We also
assume, for sake of notational convenience, that there is at most 1 element of
$E^{crit}$ in each fibre.

\item For all $z \notin S^{crit}$ the fibre $E_z = \pi^{-1}(z)$ becomes a
symplectic manifold with convex boundary with respect to $\o|_{E_z}$.  This
means that we get a splitting of tangent spaces
\[
TE_x = TE^h_x \oplus TE^v_x,
\]
where the vertical part $TE^v_x$ is the kernel $\ker(D\pi_x)$ and the
horizontal
part $TE^h_x$ is the orthogonal complement of $TE^v_x$ with respect to $\o$.

\item At every point $x \in E$ such that $z=\pi(x) \in \partial S$, we have $TS
= T(\partial S) + D \pi (TE_x)$.  This implies that $\pi^{-1}(\partial S)$ is a
boundary stratum of $E$ of codimension 1, which we shall call the
\emph{vertical boundary}, denoted $\partial^v E$.  The union of boundary faces
of $E$ not contained in $\partial^v E$ we shall call the \emph{horizontal
boundary} of $E$, denoted $\partial ^h E$.

\item If $F$ is a boundary face of $E$ not contained in $\partial ^v E$, then
$\pi|_F \colon F \to S$ is a smooth fibration, which implies that any fibre is
smooth near its boundary.  We also want the horizontal boundary $\partial^h E$
to be horizontal with respect to the above splitting, so that parallel
transport (see below) will be well-defined along the boundary.
\end{itemize}

\begin{defn}
\label{defn:lef}
If all the above holds we call $(E,\pi,\o)$ a compact convex Lefschetz
fibration.  For ease of notation, in what follows we shall often call
$(E,\pi,\o)$ simply a Lefschetz fibration, suppressing the extra adjectives.
\end{defn}

The splitting of tangent spaces into horizontal and vertical subspaces means
that we have a connection
over $S \setminus S^{crit}$, and hence \emph{symplectic parallel
transport} maps.  In other words, for a path $\gamma \colon [0,1] \to S$ which
misses
$S^{crit}$, our connection defines a symplectomorphism $\phi_{\gamma} \colon
E_{\gamma(0)} \to
E_{\gamma(1)}$.

There is a method \cite{mcleanlefschetz} of completing $E$ to a symplectic
manifold $\widehat{E}$ which is convex at infinity, such that we get a map
$\widehat{\pi} \colon \widehat{E} \to \widehat{S}$ to the completion of the
base.  When $S$ is a disc $\DD$, this is done as follows: firstly, the horizontal boundary $\partial^h
E$ is just $\partial M \times \DD$, where $M$ is a smooth fibre, and we can attach $\partial M \times
[0,\infty) \times \DD$ to $\partial^h E$ in the same as we complete a symplectic
manifold with convex boundary.  This gives us a new manifold we shall call
$E_1$ and we can extend $\pi$ to $\pi_1$ on $E_1$ in the obvious way.  Now
consider $\pi_1^{-1}(\partial \DD)= N$.  Attach to this $N \times [0,\infty)$ and
call the resulting manifold $\widehat{E}$, over which we can extend $\pi_1$ to
$\widehat{\pi}$.  More details can be found in \cite[Section
2]{mcleanlefschetz}.  This map $\widehat{\pi}$ restricts to $\pi$ on the
subsets corresponding to $E$ and $\DD$ and outside we have a local model looking
like the completion of the mapping cone for some symplectic map $\mu$ which we
shall call the \emph{outer monodromy} of the Lefschetz fibration $E$.  Given
this, we shall also talk in this paper about Lefschetz fibrations over $\CC$,
which are understood to be the completions of Lefschetz fibrations over some
disc $\DD_R \subset \CC$, in the sense of Definition \ref{defn:lef}.

\subsection{Vanishing cycles}
We can use the parallel transport maps to introduce the notion
of a \emph{vanishing cycle}. Choose an embedded path $\gamma \colon [0,1] \to
S$ such
that $\gamma^{-1}(S^{crit}) = \{1\}$. We can consider the set of points which
tend
to the critical point $y=\gamma(1)$ under our parallel transport maps
\[
V_{\gamma} =
\left\{  x \in E_{\gamma(0)} : \lim_{t \to 1}
{\phi_{\gamma\left.\right|_{[0,t]}}}(x)=y  \right\}.
\]
This set $V_{\gamma}$ is called the \emph{vanishing cycle} associated to the
\emph{vanishing path} $\gamma$. The vanishing cycle is actually a Lagrangian
sphere in the fibre \cite{seidelles} and if we take the \emph{Lefschetz
thimble}, the union of the images of the vanishing cycle as we move along
$\gamma$ together with the critical point, we get a Lagrangian ball
$\Delta_{\gamma}$ in the total space $E$.  In fact, $\Delta_{\gamma}$ is the
unique embedded Lagrangian ball that lies over $\gamma$. These vanishing cycles
come together with the extra datum of a ``framing'' \cite[Lemma 1.14]{seidelles}, meaning a parametrization
$f \colon S^n \to V$.  Here, two framings $f_1,f_2$ are equivalent if $f_2^{-1}f_1$ can be deformed inside $Diff(S^n)$ to an element of $O(n+1)$, but this framing information is irrelevant in the
dimensions in which we work, so shall neglect to mention framings in what
follows.

\subsection{Constructing Lefschetz fibrations}
Given a Lefschetz fibration $(E,\pi)$, we can pick a smooth reference fibre
$E_z$ and a collection of vanishing paths $\gamma_i$, one for each critical
point,
which all finish at $z$, but which are otherwise disjoint.  This then gives us
a symplectic manifold $M=E_z$ and a
collection of vanishing cycles $V_i \subset M$ associated to the $\gamma_i$.
For our purposes, in
constructing symplectic manifolds, it is important to note that we can go
the other way as in the following lemma, taken from \cite[Lemma
16.9]{seidelbible}
but with unnecessary assumptions of exactness removed.

\begin{lem}
\label{thm:build}
Suppose we have a collection $(V_1,\ldots,V_m)$ of (framed) Lagrangian spheres in a symplectic manifold $M$ with
convex boundary. On the disc $\DD$, choose a base point $z$,
and a distinguished basis of vanishing paths $\gamma_1,\ldots,\gamma_m$ all
of which have one endpoint at $z$. Then there is a compact convex Lefschetz
fibration $\pi
\colon E \to \mathbb{D}$, whose critical values are precisely the endpoints
$\gamma_1(1),\ldots,\gamma_m(1)$; this comes with an identification
$E_z=M$, under which the (framed) vanishing cycles $V_{\gamma_k}$
correspond to $V_k$.
\end{lem}

This will be the technique used to construct the symplectic manifolds
considered in this paper. However, in order to do this, we need
to identify a collection of Lagrangian spheres in a given symplectic manifold
$M$. In the case where $M$ itself admits a Lefschetz fibration, we shall do
this by considering
matching cycles.

\subsection{Matching cycles}
Consider a Lefschetz fibration $\pi \colon M \to S$ and an embedded path
$\gamma \colon [0,1] \to S$ such that $\gamma^{-1}(S^{crit}) = \{0,1\}$. In
the fibre $\pi^{-1}(\gamma(\frac{1}{2}))$, we get two vanishing cycles, one
coming from either endpoint. If they agree, then parallel transport allows
us to glue the two thimbles together to obtain a smooth Lagrangian sphere
$V \subset M$. We shall call $\gamma$ a \emph {matching path}, and $V$ the
associated \emph{matching cycle}.

In this paper we shall usually work in situations where the vanishing cycles do
agree exactly so that we do get matching cycles, but
occasionally we will have the situation where the two vanishing cycles are
not equal, but are merely Hamiltonian isotopic.  In this situation we may
appeal to the following result of
\cite[Lemma 8.4]{amp}:

\begin{lem}
\label{thm:amp}
Let $(M,\omega)$ be a symplectic manifold with a Lefschetz
fibration $\pi \colon M \to \CC$ and let $\gamma \colon [0,1] \to \CC$ be a
path such that $\gamma^{-1}(S^{crit}) = \{0,1\}$.  Suppose that the two
vanishing cycles $V_0,V_1 \subset M_{\gamma(\frac{1}{2})}$ coming from either
end of this path are Hamiltonian isotopic for some compactly supported
Hamiltonian $H_s$ defined on the fibre $M_{\gamma(\frac{1}{2})}$.  Then $M$
contains a Lagrangian sphere.
\end{lem}

Matching cycles will be used for our main method of construction. We take a
symplectic manifold $\mo$ equipped with a Lefschetz fibration and consider
an ordered collection of matching paths. In favourable circumstances these
will give rise to a family of framed Lagrangian spheres $(V_1, \ldots ,
V_n) \subset M$ and we now apply Lemma \ref{thm:build} to construct a new
Lefschetz
fibration $(E,\pi)$.

\subsection{Maydanskiy's examples}
Figure \ref{fig:maydanskiy} shows the examples considered in \cite{maydanskiy}.
Although higher-dimensional examples are also considered in \cite{maydanskiy},
the meaning of all such diagrams in this paper is that we take the
symplectic manifold $M^4$ built according to Lemma \ref{thm:build} by taking
fibre $T^*S^1$ and vanishing cycles given by the zero-section, one for each
cross. The lines in Figure \ref{fig:maydanskiy} are then matching paths which
yield the spheres
required to apply Lemma \ref{thm:build} again to obtain $E^6$.  The fact that
the paths in Figure \ref{fig:maydanskiy} actually do give matching cycles will
for us be a consequence of the method of construction considered in the next
section.

\begin{figure}[h!]
\[
\xy
(35,0)*{\xy
(-10,0)*{\times}; (10,0)*{\times}; **\crv{(0,13)};
(-10,0)*{}; (10,0)*{}; **\crv{(0,-13)};
(0,0)*{\times};
\endxy};
(-35,0)*{\xy
(-10,0)*{\times}; (10,0)*{\times}; **\crv{(0,13)};
(-10,0)*{}; (10,0)*{}; **\crv{(0,-13)};
(0,10)*{\times};
\endxy};
(35,-14)*\txt{$X_1$ -contains no Lagrangian $S^3$};
(-35,-14)*\txt{$X'_1$ -contains a Lagrangian $S^3$};\endxy
\]
\caption{}
\label{fig:maydanskiy}
\end{figure}

Maydanskiy \cite{maydanskiy} proves that the two symplectic manifolds in Figure
\ref{fig:maydanskiy} are diffeomorphic (they are both $T^*S^3 \cup 2$-handle)
but are not symplectomorphic.  $X'_1$ is just $T^*S^3$ with a Weinstein
2-handle attached as in \cite{weinsteinhandle} and contains an exact Lagrangian
sphere inherited from the zero-section of $T^*S^3$.  In contrast, $X_1$
contains no exact Lagrangian submanifolds, and so is considered exotic.

One way of thinking about this intuitively is that
the manifolds are diffeomorphic because one can construct a smooth isotopy
taking
the top matching cycle in $X_1$ and moving it over the critical
point in the middle to get $X'_1$.  The reason this fails to work
symplectically is that we are free to move our cycles only by Hamiltonian
isotopies,
and we will not then be able to avoid the central critical point (since
we cannot displace the zero-section of $T^*S^1$), although the actual proof in
\cite{maydanskiy} has to make use of more sophisticated Floer-theoretic
arguments.

\section{Deformations of symplectic structures}
\label{sec:maydanskiy}

\begin{defn}
Let $(E,\o)$ be a symplectic manifold.  By a deformation of the symplectic
structure $(E,\o)$ we shall mean a
smooth 2-form $\Omega$ on $\tilde{E}=E
\times [0,1]$ such that
\begin{itemize}
\item $\Omega|_t$ is symplectic on each $E \times \{t\}$
\item $\Omega|_0=\o$
\item $\iota_v\Omega=0$ for any $v \in \ker(D\rho)$ where $\rho$ is the
projection $\tilde{E} \to E$.
\end{itemize}
This is equivalent to a smooth 1-parameter family of symplectic forms
$\left\{\o_t : t\in[0,1] \right\}$ on $E$ such that $\o_0=\o$.  We shall denote
by $(\tilde{E}^t,\o_t)$ the symplectic manifold $(E \times \{t\},\Omega|_t)$.
\end{defn}

We shall consider $X_1$, the exotic example of Maydanskiy from the previous
section. In this section, we shall prove

\begin{thm}
\label{thm:maydanskiy}
There is a deformation $\tilde{X}_1$ of $X_1$ such that, for all $t >0$,
$\tilde{X}_1^t$ contains a Lagrangian sphere.
\end{thm}

\subsection{\texorpdfstring{Constructing a deformation of $X_1$}{Constructing a
deformation of X1}}
The fibres of Maydanskiy's examples are $A_2$ Milnor fibres.  For our purposes,
which crucially rely on matching paths defining genuine matching cycles without
having to rely on Lemma \ref{thm:amp}, we shall work with the specific model as
below.

Let $M$ be the affine variety defined by
\[
M = \left\{ z_1^2 + z_2^2 = (z_3-1)(z_3-2)(z_3-3) \right\} \subset \CC^3
\]
equipped with symplectic form $\o$, which is the restriction of the standard
symplectic form on $\CC^3$.  We may restrict to some compact set $M^{in}
\subset M$ ($M^{in} \subset B_R \subset \CC$ for some sufficiently large $R$),
such that $M^{in}$ is a Liouville domain which becomes a Lefschetz fibration in
the sense of Definition \ref{defn:lef} once we project onto the
$z_3$-coordinate \cite[Section 19b]{seidelbible}.  It has three critical values, at 1, 2 and 3.

There is a homologically essential Lagrangian sphere $A$ living over the
straightline path joining the two critical points at 1 and 2, which is given by
the part of the real locus $M_{\RR}$ living over this path.  This sphere is
precisely the matching cycle associated to that line.  We can do the same with
the part of $M \cap \RR \langle x_3, y_1, y_2 \rangle$ living over the interval
$[2,3]$ to find another Lagrangian sphere $B$ and we shall take $A$ and $B$ to
define our standard basis of $H_2(M;\RR) = \RR^2$.

The manifold $M$ carries an $S^1$-action given by
\[
\left( \begin{array}{ccc}
z_1\\
z_2\\
z_3 \end{array}
\right)
\mapsto
\left( \begin{array}{ccc}
\cos \theta & -\sin \theta & 0  \\
\sin \theta & \cos \theta & 0  \\
0 & 0 & 1  \end{array}
\right)
\left( \begin{array}{ccc}
z_1\\
z_2\\
z_3 \end{array}
\right)
\]
and the symplectic form  $\o$ is invariant under this action.

Every smooth fibre is of the form ${ z_1^2 + z_2^2 = \lambda}$ for some nonzero
$\lambda = se^{i\alpha}$  and we observe that such a fibre is preserved by the
$S^1$-action, which in particular means that the parallel transport map
associated to a path $\gamma$ is $S^1$-equivariant.  This fibre is
symplectomorphic to $T^*S^1$, where the model we use for $T^*S^1$ is
\[
T^*S^1=\left\{(q,p) \in \RR^2 \times \RR^2 :  \left\|q\right\|=1 \; , \;
\langle q,p \rangle =0   \right\}.
\]
The symplectomorphism is defined as follows: let $\hat{z} = z e^{-i \alpha /
2}$ and map
\[
 z \mapsto \left( \frac{\Re(\hat{z})}{\|\Re(\hat{z})\|}, -\Im(\hat{z})
\|\Re(\hat{z})\| \right).
\]
Note that, for each fibre, the $S^1$-orbits are mapped to level sets $\|p\|=
constant$ so, given that the parallel transport maps are $S^1$-equivariant, the
vanishing cycle associated to any vanishing path will itself correspond to such
a level set.

We shall deform the symplectic structure by introducing 2-forms which are intended to resemble area forms supported near the equators of $A$ and $B$.  We therefore consider the 2-form on $\CC^3 \setminus i\RR^3$,
\[
\eta = g_{\epsilon}\left(\frac{x}{\|x\|} \right) \left(  \frac{x_1}{\|x\|^3}
dx_2 \wedge dx_3 + \frac{x_2}{\|x\|^3} dx_3 \wedge dx_1 + \frac{x_3}{\|x\|^3}
dx_1 \wedge dx_2 \right)
\]
where $g_{\epsilon}(x)= g_{\epsilon}(x_3)$ denotes a cutoff function for the
$x_3$-coordinate which has $\supp(g_{\epsilon}) \subset \{ |x_3|  <
\epsilon\}$.

As $\eta$ is defined using only coordinates on the real slice $\RR^3 \setminus
\{0\}$ and annihilates the radial direction, this is a closed form on $\CC^3
\setminus i\RR^3$.  We shall choose $\epsilon$ such that $\epsilon < \frac{1}{8R}$, and apply a translation $x \mapsto x +(0,0,3/2)$.  It is easy to show that $\eta$ is now well-defined on $M$, so that in the Lefschetz fibration $M^{in}
\to \DD_R$, $\eta$ is a closed, $S^1$-equivariant 2-form supported in the region lying over $\{|x_3 - 3/2| <1/4\}$ and the sphere $A$
has some nonzero area with respect to $\eta$.

Moreover, we can rescale $\eta$ so that $\o + \eta$ is still symplectic on
$M^{in}$, since the property of being symplectic is an open condition and
$M^{in}$ is compact.  Also, since $M$ is an $A_2$ Milnor fibre, its boundary
$\partial M$ is topologically the quotient of $S^3$ by a $\ZZ/3$ action and
therefore $H^2(\partial M ; \RR)=0$.  This means that, perhaps after rescaling
$\eta$ again, $M^{in}$ will still have contact boundary.

We repeat the above procedure to obtain another closed 2-form $\eta'$ on
$M^{in}$, defined now using the coordinates $y_1,y_2,x_3$ which is again
$S^1$-equivariant and is supported over $\{|x_3 - 5/2| < 1/4 \}$ and has the
property that
\[
\eta(A) = -\eta'(B).
\]
We denote by $\o_t$ the 2-forms $\o + t(\eta + \eta')$ for $t \in [0,1]$, all of
which make $M^{in}$ symplectic with convex boundary.

\begin{rmk}
Such a construction can be generalized: choose a finite collection of distinct
points $p_1, \ldots , p_{n+1} \in \RR$  and consider the affine variety
\[
M_{\mathbf{p}} = \left\{ z_1^2 + z_2^2 = \prod_i(z_3-p_i) \right\} \subset
\CC^3,
\]
which will be diffeomorphic to the $A_n$ Milnor fibre, with a basis of
$H_2(M_{\mathbf{p}})$ given by the spheres $A_i$ living over the straightline
path joining $p_i$ and $p_{i+1}$.  We may construct a deformation of the
symplectic structure on $M_{\mathbf{p}}$ by adding on 2-forms which are
supported on strips lying between the critical points as above.
\end{rmk}

\subsection{Obstructions to forming matching cycles are purely homological}
We now consider the path $\gamma_0$ in Figure \ref{fig:homology}, going from 1
to 3 in $\CC$.  We would like this to define a genuine matching cycle, with
respect to the parallel transport maps coming from $\o_t = \o + t(\eta +
\eta')$ for $t \in [0,1]$.  However, we may no longer get a genuine Lefschetz
fibration in the sense of Section \ref{sec:lef}, since the horizontal boundary
may no longer be horizontal with respect to our splitting.  This means that
parallel transport cannot be done near $\partial^h M$, but we shall not need
this: our vanishing cycles stay within a region away from the boundary, since
deforming the symplectic form will only change the parallel transport maps by a
small amount.

\begin{figure}[h!]
\[
\xy
(20,60)*\xycircle(8,2){-};
(50,60)*\xycircle(8,2){-};
(80,60)*\xycircle(8,2){-};
(20,20)*\xycircle(8,2){-};
(50,20)*\xycircle(8,2){-};
(80,20)*\xycircle(8,2){-};
(12,60)*{}; (28,20)*{} **\crv{};
(28,60)*{}; (12,20)*{} **\crv{};
(72,60)*{}; (88,20)*{} **\crv{};
(88,60)*{}; (72,20)*{} **\crv{};
(42,60)*{}; (42,20)*{} **\crv{(48,40)};
(58,60)*{}; (58,20)*{} **\crv{(52,40)};
(20,0)*{\times};
(20,-3)*{1};
(80,0)*{\times};
(80,-3)*{3};
(61,9)*{\gamma_0};
(62,-9)*{\gamma_1};
(50,0)*{\times};
(50,-3)*{2};
(20,0)*{}; (50,0)*{}; **\dir{.};
(50,0)*{}; (80,0)*{}; **\dir{.};
(60,2)*{B};
(40,2)*{A};
(30,9)*{V^t_0};
(28,-8)*{V^t_1};
(20,0)*{}; (80,0)*{} **\crv{(50,15)};
(20,0)*{}; (80,0)*{} **\crv{(50,-15)};
(50,46)*\xycircle(5.5,1.5){-};
(50,34)*\xycircle(5.5,1.5){-};
(50,40)*{S_t};
(20,40)*{}; (50,47.5)*{} **\crv{(20,43)};
(80,40)*{}; (48,35.5)*{} **\crv{(80,43)};
(20,40)*{}; (52,44.5)*{} **\crv{(20,37)};
(80,40)*{}; (50,32.5)*{} **\crv{(80,37)};
\endxy\]
\caption{}
\label{fig:homology}
\end{figure}

Therefore, for any given $t$, the path $\gamma_0$ gives us two circles in the
central fibre which we know correspond to level sets $\|p\|=constant$.  (In
Figure \ref{fig:homology}, the fibres shown at the top are those living over
the path $\gamma_0$.) These two circles enclose some chain $S_t$ in the fibre
over $\gamma_0(\frac{1}{2})$, and the sum of this chain and the two thimbles is
homologous to $[A] +[B]$, so therefore has symplectic area 0 with respect to
$\o_t$.  Since the vanishing thimbles are Lagrangian, this means that the chain
$S_t \subset T^*S^1$ must also have zero symplectic area, and therefore $S_t$
must in fact be empty.  In other words, we get a genuine matching cycle for all
$t$, which we denote $V^t_0$.  We can do likewise for the path $\gamma_1$ to
obtain the matching cycle $V_1^t$.

By the same argument, for any $t>0$ we can take a straightline path given by
the interval $[1,3]$, which goes over the central critical point at 2, and say
that this too will define a matching cycle: in the central nonsmooth fibre we
shall either get, by $S^1$-symmetry, the critical point or some circle.
However, if we obtained the critical point, then we would have found a
Lagrangian in a homology class of positive symplectic area.  Which smooth
component this circle lives in depends on whether we choose to give the class
$A$ positive or negative area.

Therefore, for $t>0$, we can take a smooth family of paths interpolating
between the two matching paths and get a smooth family $(V^t_s)_{s \in [0,1]}$
of Lagrangian $S^2$s
joining the two matching cycles.  This has the following standard consequence.

\begin{lem}
\label{thm:hamiltopic}
For $t>0$, $V^t_0$ and $V^t_1$ are Hamiltonian isotopic.
\end{lem}

\begin{proof}
We can identify some neighbourhood of $V^t_0$ with $T^*S^2$ and, for $0 \leq s
\leq s_0$ for some small $s_0$, $V^t_s$ will correspond to the graph of some
1-form $\alpha_s$.  Since $V^t_s$ is Lagrangian, $d\alpha_s=0$, and therefore
$\alpha_s=df_s$ since $H^2(V^t_s;\RR)=0$.  We can moreover choose these $f_s$
smoothly.  A direct calculation shows that $H(x,s)= \frac{d}{ds}(f_s(\rho(x)))$
is a Hamiltonian yielding an isotopy between $V^t_0$ and $V^t_{s_0}$, where
here $\rho \colon T^*S^2 \to S^2$ is the standard projection map.  We can patch
together such isotopies to get from $V^t_0$ to $V^t_1$, and then apply some
cutoff function to make our Hamiltonian to be compactly supported.
\end{proof}

\subsection{\texorpdfstring{$X_1$ contains a Lagrangian sphere after
deformation}{X1 contains a Lagrangian sphere after deformation}}
We are now in a position to prove Theorem \ref{thm:maydanskiy}.  To do this, we
shall establish a deformation version of Lemma \ref{thm:build}.  This is stated
below in the case where there is just one vanishing cycle, since the general
case follows from gluing together such examples.

Suppose we have $\tilde{M}$, a deformation of the symplectic structure $(M,\o)$, such that
$\tilde{M}^t$ has convex boundary for all $t$, and suppose that we also have
$\tilde{V} \subset \tilde{M}$, which is the image of an embedding of $S^n \times[0,1]$ such that, for
all $t$, we get a Lagrangian sphere $\tilde{V}^t \subset (\tilde{M}^t, \o_t)$.

Then, by Lemma \ref{thm:build}, we can construct a Lefschetz fibration $E^t \to \DD$ from $\tilde{M}^t$ and $\tilde{V}^t$ for each $t$.  We want the family $E^t$ to comprise a deformation of $(E^0, \o_E)$.

\begin{prop}
\label{thm:deformbuild}
In the above situation, we can construct a bundle of symplectic manifolds $\tilde{E} \to
[0,1]$, such that each fibre $\tilde{E}^t$ has convex boundary and comes with
an identification $\tilde{E}^t_z \cong \tilde{M}^t$, under which the vanishing
cycle $V_{\gamma}$ corresponds to $\tilde{V}^t$.  After applying a
trivialization of this bundle which is the identity over 0, this is a
deformation of $(E^0,\omega_E)$.
\end{prop}

\begin{proof}
We closely follow \cite[Proposition 1.11]{seidelles}.  First we need a neighbourhood theorem, whose proof follows the same argument as that of the standard Lagrangian neighbourhod theorem \cite{mcduffsalamonintro}.

\begin{lem}
\label{thm:nbhd}
Let $(\tilde{M},\Omega)$ be a deformation of $\mo$.  Suppose we have $\tilde{V}
\subset \tilde{M}$ an embedding of $V \times [0,1]$ such that, for all $t$, we
get a Lagrangian  $\tilde{V}^t \subset (\tilde{M}^t,\o_t)$.  Then there exists
a neighbourhood $\mathcal{N} \subset T^*V \times [0,1]$ of the zero-section $V
\times [0,1]$ and a neighbourhood $\mathcal{U} \subset \tilde{M}$ of
$\tilde{V}$ and a diffeomorphism $\phi \colon \mathcal{N} \to \mathcal{U}$ such
that $\phi^*\Omega = \beta$ where $\beta$ is the 2-form on $T^*V \times [0,1]$
given by the pullback of the standard symplectic form on $T^*V$.
\end{lem}

In our case, we may assume our neighbourhood $\mathcal{N}$ in Lemma
\ref{thm:nbhd} is of the form $\mathcal{N}=T^*_{\leq \lambda} S^n \times [0,1]$
for some $\lambda >0$, where $T^*_{\leq \lambda}S^n$ denotes the disc cotangent
bundle with respect to the standard metric on $T^*S^n$.  Given this, we follow
\cite[Proposition 1.11]{seidelles} which starts by considering the local
Lefschetz model $q \colon \CC^{n+1} \to \CC$, $q(z) = \sum{z_i^2}$.  We also
consider the function $h(z) = \|z\|^4 - |q(z)|^2$.

When we restrict to $W \subset \CC^{n+1}$ cut out by the inequalities $h(x)
\leq 4 \lambda^2$ and $|q(z)| \leq 1$, we get a compact Lefschetz fibration
$\pi_W \colon W \to \DD$.  As explained in \cite{seidelles}, $W$ comes together
with an identification $\psi \colon \pi_W^{-1}(1) \to T_{\leq \lambda}^*S^n$, a
neighbourhood $Y\subset W$ of $\partial_h W$, a neighbourhood $Z$ of $\partial
(  T_{\leq \lambda}^*S^n)$ in  $T_{\leq \lambda}^*S^n$ and a diffeomorphism
$\Psi \colon  Y \to \DD \times Z$ which fibres over $\DD$ and agrees with
$\psi$ on $Y \cap \pi_W^{-1}(1)$.  Let $\tilde{W}= W \times [0,1]$ and, by
taking the product with $[0,1]$, consider the corresponding
$\tilde{Y},\tilde{Z},\tilde{\psi},\tilde{\Psi}$.

Now define $\tilde{M}_{-}$ to be $\tilde{M} \setminus (\phi(\mathcal{N}
\setminus \tilde{Z}))$ and consider
\[
\tilde{E} = \DD \times \tilde{M}_{-} \cup_{\sim} \tilde{W},
\]
where the identification made identifies $\tilde{Y}$ with $\DD \times
\phi(\tilde{Z})$ through $(id \times \phi) \circ \tilde{\Psi}$.  This now has
all the required properties.
\end{proof}

\begin{proof}[of Theorem~{\rm\ref{thm:maydanskiy}}]
Using Proposition \ref{thm:deformbuild}, we can construct a deformation
$\tilde{X}_1$ of Maydanskiy's exotic example $X_1$ and we want to say that we
have a Lagrangian sphere $L^t \subset \tilde{X}^t_1$ for all $t>0$.
$\tilde{X}^t_1$ admits a Lefschetz fibration with two critical points.  We take
a path joining the two critical points in the Lefschetz fibration on $X_1$.  If
we choose the vanishing paths $\gamma$ in Proposition \ref{thm:deformbuild}
such that they join together smoothly, then the concatenation of these paths is
smooth and yields two vanishing cycles in the central fibre, which are
precisely just $V^t_0$ and $V^t_1$ from Lemma \ref{thm:hamiltopic}, which we
know are Hamiltonian isotopic for all $t>0$.  We then just apply Lemma
\ref{thm:amp} to find a Lagrangian sphere.
\end{proof}

\begin{rmk}
As $t \to 0$, the Lagrangian spheres $L^t$ degenerate to some singular
Lagrangian cycle, which is worse than immersed.  In fact, topologically it
looks like $S^3$ with some $S^1$ in it collapsed to a point.  Presumably,
pseudoholomorphic curve theory with respect to this cycle is very badly
behaved, so that a Floer theory along the lines of \cite{joyceimmersed} cannot
be made to work here, although see \cite{joycecounting} for some analysis of holomorphic discs on certain similar special Lagrangian cones.
\end{rmk}

\section{Floer cohomology}
\label{sec:floer}
To consider $X_2$ and adapt the arguments presented in \cite{maydanskiyseidel},
we shall need to consider the Lagrangian Floer cohomology $HF(L_0,L_1)$ of two
transversely intersecting Lagrangian submanifolds in some symplectic manifold
$\mo$.  To define this, one has to pick a generic family of almost complex
structures $\bj=(J_t)$, which are usually required to be compatible with $\o$,
in the sense that $g_t(u,v) = \o(u,J_t v)$ defines a Riemannian metric.
However, we shall want to consider $J_t$ which are $\o$-tame except on a small
neighbourhood of $L_0 \cap L_1$, where here $J_t$ is still $\o$-compatible.
($\o$-tame means that $\o(u,J_t u)>0$ for all nonzero $u$.)  We shall show that,
given any such family of almost complex structures $\bj=(J_t)$, there exists
$\mathbf{\tilde{J}}=(\tilde{J}_t)$ arbitrarily close to it, with the same
properties, such that $HF(L_0,L_1)$ can be defined with respect to
$(\tilde{J}_t)$.  The key point is that we are using Cauchy-Riemann type
operators with totally real boundary conditions, so all the relevant elliptic
regularity theory can still be applied.

\begin{rmk}
The content of this section, that we can relax the condition on the almost
complex structures to define Floer cohomology is probably already known to
experts, but we are unaware of any written account of this in the literature.
\end{rmk}

\subsection{Setup}
Let $(M,\o)$ be a symplectic manifold of dimension $2n$ and let $L_0$, $L_1$ be
two Lagrangian submanifolds which intersect transversely.  For each
intersection $x$, fix some small open set $U_x$ around $x$ such that $L_0 \cap
L_1 \cap U_x = \{x\}$.  Assume moreover that the $U_x$ are disjoint.  Pick some
family $\bj=(J_t)$ of smooth almost complex structures which tame $\o$ (this in
particular implies that the $L_k$ are totally real), and which are
$\o$-compatible on each $U_x$.

We note here for future reference the following lemma due to Frauenfelder
\cite{frauenfelder}.

\begin{lem}
\label{thm:frauenfelder}
Let $(M^{2n},J)$ be an almost complex manifold and $L^n \subset M$ a totally
real submanifold.  Then there exists a Riemannian metric $g$ on $M$ such that

\begin{itemize}
\item $g(J(p)v,J(p)w) = g(v,w)$ for $p \in M$ and $v,w \in T_pM$,
\item $J(p)T_pL$ is the orthogonal complement of $T_pL$ for every $p\in L$,
\item $L$ is totally geodesic with respect to $g$.
\end{itemize}

\end{lem}

Let $\Sigma$ denote the holomorphic strip $\RR \times [0,1] \subset \CC$.
Given a map $u \colon \Sigma \to M$, we can consider the $\dbar$ operator
defined by
\[
\dbar u(s,t) = \partial_s u(s,t) + J_t(s,t) \partial_t u(s,t).
\]
We care about holomorphic maps, which are just those such that $\dbar u=0$ and
we define the energy of any map $u$ to be $E(u) = \int {\| \partial_s u \|^2
}$.

Let $\MM_{\bj}$ denote the set of holomorphic $u$ as above which also satisfy
the boundary conditions $u(s,0) \in L_0$, $u(s,1) \in L_1$ as well as $E(u) <
\infty$.  It is proved in \cite{robbinsalamon} that any such map must have the
property that
\[
\lim_{s \to \pm \infty} u(s,t) = x^{\pm},
\]
where $x^{\pm}$ are intersection points in $L_0 \cap L_1$.  Moreover, the
convergence near the ends is exponential in a suitable sense about which we
shall say more later.  We define $\MM_{\bj}(x,y)$ to be the space of
finite-energy trajectories as above which converge to $x$ and $y$ at the ends.

We want to examine the properties of $\MM_{\bj}(x,y)$ and, in particular,
determine when it is a smooth manifold, so we follow the standard procedure of
Floer \cite{floergradientflow}, in exhibiting $\MM_{\bj}(x,y)$ as the zero set
of some Fredholm section of a Banach bundle.  Much of what follows is already
contained in Floer's original work \cite{floergradientflow}, but we shall
recall the main details for the reader's convenience.

\subsection{Banach manifolds}
Let $kp>2$. We can consider the Sobolev space $L^p_{k;loc} (\Sigma,M)$ and
define
\[
\ppk = \left\{ u \in \lpk(\Sigma,M) : u(s,0) \in L_0 , \; u(s,1) \in L_1
\right\}.
\]
Let $\Sigma_{\rho} = \{z \in \Sigma : |\Re z| < \rho \}$.  The topology on
$\ppk$ is defined using the basis of open sets given by
\[
 \mathcal{O}_{u,\rho,\epsilon} = \left\{ v \in \ppk : v = \exp_u \xi
\mbox{ on } \Sigma_{\rho} \mbox{ and } \|\xi\| _{k,p} <\epsilon \mbox{ for } p
< \rho
\right\}.
\]
Here $u \in \ppk$ and $\rho, \epsilon >0$.

For our present purposes, and in order to ensure that we do in fact get a
Banach manifold, we shall need to restrict to a subset of $\ppk$ with nice
behaviour near intersection points $x \in L_0 \pitchfork L_1$.  Consider
\[
 \ppk(\cdot,x)= \left\{u \in \ppk : \exists \rho >0, \exists \xi \in \lpk
(\Sigma, T_xM), u(s,t) = \exp_x \xi(s,t) \forall s>\rho \right\}.
\]
In other words, we restrict attention to maps which, at one end, look like the
exponentiation of some vector field.  We impose a similar condition at the other end
to define $\ppk(x,\cdot)$, and then consider $\ppk(x,y)$.

For $u \in \ppk$, $u^*TM$ is an $\lpk$-bundle, so we can talk about
sections which are locally of $\lpk$-type.  We shall introduce the shorthand
$L^p_k(u) = \lpk(u^*TM)$
and we may also consider
\[
W^p_k(u) = \left\{ \xi \in L^p_k(u) : \xi(s,0)
\in T_{u(s,0)}L_0, \xi(s,1) \in T_{u(s,1)}L_1 \right\},
\]
so here we have tangent pointing along the Lagrangian boundary.

We can also consider spaces of sections $W^q_l(u)$ and $L^q_l(u)$ of different
regularity provided that $l \leq k$ and
\begin{equation}
\label{eqn:sobolev}
l-\frac{2}{q} \leq k - \frac{2}{p}.
\end{equation}

\begin{thm}
\emph{(\cite[Theorem 3]{floergradientflow})}
Let $p \geq 1$ and $kp>2$.  Then $\ppk(x,y) $ is a smooth Banach manifold and
its tangent space at $u$ is given by
$T_u \ppk (x,y) = W^p_k (u)$.
\end{thm}

To show this is a Banach manifold, Floer uses a system of charts based on the
exponential map.  Accordingly, pick a family of metrics $(g_t)$ such that $L_k$
is totally geodesic with respect to $g_k$, as in Lemma \ref{thm:frauenfelder}.

Define
\begin{align*}
&\exp \colon \Sigma \times TM \to M, \\
&\exp(s,t,x,v) = \exp_{g_t}(x,v).
\end{align*}

Let $\iota$ denote the minimal injectivity radius of the metrics $g_t$ and
define
\[
U_u = \left\{\xi \in W^p_k(u) : \|\xi\|_{\infty} < \iota \right\}.
\]
On a noncompact manifold $M$, we will not necessarily have $\iota >0$.
However, in our cases, this will hold since all our symplectic manifolds are
geometrically bounded at infinity.

The charts are now given by
\begin{align*}
&\exp_u \colon U_u \to V_u =\exp_u(U_u),\\
&\exp_u(\xi)(s,t) = \exp(s,t, u(s,t),\xi(s,t)).
\end{align*}

It is because of this system of charts that we restricted the convergence
conditions at the ends in defining $\ppk(x,y)$.
The proof of above theorem is technical but makes no use of the symplectic
structure.

Moreover we may also consider Banach bundles $\WW^q_l \to \ppk(x,y)$ and
$\LL^q_l \to \ppk (x,y)$, with fibres modelled on $W^q_l(u)$ and $L^q_l(u)$
respectively,  provided that the regularity condition \eqref{eqn:sobolev}
holds.

The same proof as in \cite{floergradientflow} shows that $\dbar$ is a smooth
section of $\LL^p_{k-1}$.  In fact, since $\dbar$ is a real Cauchy-Riemann
operator with totally real boundary conditions \cite[Appendix
C]{mcduffsalamonjholo} $\dbar$ is a Fredholm operator.  We denote its
linearization at $u$ by $E_u = D\dbar (u) \colon W^p_k \to L^p_{k-1}$.

We now consider the zero-set of the section $\dbar$.  It is shown in
\cite{robbinsalamon} that if $u \in \MM(x,y)$, then $u$ has the right
convergence conditions at the ends to be an element of $\PP^p_k(x,y)$ and
moreover these sets are locally homeomorphic.  Moreover, any solution to $\dbar
u=0$ will in fact be smooth, using elliptic bootstrapping techniques.  This is
proved in \cite{floergradientflow} for $\o$-compatible $\bj$, and this proof
carries over in region $U_x$, and elsewhere it follows from \cite[Proposition
3.1.9]{mcduffsalamonjholo}.  Therefore the zero set of $\dbar$ is precisely
$\cup_{x,y} \MM_{\bj}(x,y)$.

\subsection{Fredholm theory}
\label{sec:fredholm}
This zero set will not always be a manifold, but we shall show that we can
always perturb $\bj =(J_t)$ to some arbitrarily close $\mathbf{\tilde{J}}=
(\tilde{J}_t)$ such that the corresponding moduli space
$\MM_{\mathbf{\tilde{J}}}$ is in fact a manifold.  To do this, we need to have
some space which represents the possible perturbations of $\bj$.

The space of $\o$-tame $J$ is a Fr\'{e}chet manifold whose tangent space at $J$
is given by smooth sections of  $\End(TM,J,\o)$, which is defined to be the
bundle over $M$ whose fibre at $x$ is the space of linear maps $Y \colon T_xM
\to T_xM$ such that $YJ +JY = 0$.  In order that we may have a Banach manifold,
not a Fr\'{e}chet one, we use the following argument of Floer
\cite{floergradientflow}.

Pick any sequence of positive real numbers $(\epsilon_k)$ and define
\[
\|Y\|_{\epsilon} = \sum \epsilon_k \max_x|D^k Y(x)|.
\]
Denote by $C^{\infty}_{\epsilon}(M,\End(TM,J,\o))$ those $Y$ with finite $\|
\cdot\|_{\epsilon}$ norm.  This is a Banach manifold.  Floer
\cite{floergradientflow} proves that there is a sequence $(\epsilon_k)$
that tends to zero sufficiently quickly that
$C^{\infty}_{\epsilon}(M,\End(TM,J,\o))$ is dense in $L^2(M,\End(TM,J,\o))$.

Now fix some 1-parameter family $\mathbf{J^0}=(J^0_t)$ of almost complex
structures. For a 1-parameter family $\mathbf{Y}=(Y_t)$ of elements of
$C^{\infty}_{\epsilon}(M,\End(TM,J,\o))$, we consider the map $f \colon Y_t
\mapsto  J^0_t \exp (-J^0_t Y_t))$.  On some neighbourhood of the zero-section
$f$ restricts to a diffeomorphism.  Define
\[
\mathcal{Y} = \left\{ \mathbf{Y} =(Y_t) \colon \|Y_t\|_{\infty}<r \mbox{ and }
Y_t(p)=0 \mbox{ for } p \in U \right\},
\]
where $U = \cup_x U_x$ is our neighbourhood of the intersection points $x$ and
$r$ is chosen small enough such that the restriction of $f$ is a
diffeomorphism.  Denote by $\JJ_{r}(\mathbf{J^0})$ the image of $\mathcal{Y}$
under $f$.  This space represents our space of perturbations of $\mathbf{J^0}$.
In what follows, we shall usually consider $\mathbf{J^0}$ to be fixed and write
$\JJ$ instead of $\JJ_{r}(\mathbf{J^0})$.

We have a section of Banach manifolds
\begin{align*}
&\tilde{\partial} \colon \PP \times \mathcal{Y} \to \LL, \\
&\tilde{\partial}(u,\mathbf{Y}) = \bar{\partial}_{f(\mathbf{Y})} u.
\end{align*}
As before, this section is smooth.  We want to prove that its linearization is
surjective on its zero set.  Since $E_u=D\dbar(u)$ is closed, it suffices to
prove that the image is dense whenever $\dbar u =0$.  This is proved in the
Appendix of \cite{oh}, which is itself a correction of the argument appearing
in \cite{floergradientflow}.  This result makes no assumption of any
$\o$-compatibility condition.

Now the implicit function theorem \cite[Theorem A.3.3]{mcduffsalamonjholo} says
that the \emph{universal Floer moduli space}
\[
\MM(x,y,\JJ) = \left\{ (u,\bj): u \in \MM_{\bj}(x,y) \right\}
\]
is a smooth Banach manifold.  Once we have this, we may consider the projection
onto the $\JJ$ factor, which is a Fredholm map and apply the Sard-Smale
theorem.

\begin{thm}[Sard-Smale]
The set of regular values of a Fredholm map $g \colon A \to B$ between
paracompact Banach manifolds is a Baire set in $B$.
\end{thm}

This shows that there is a second category set $\JJ_{reg} \subset \JJ$ of
so-called \emph{regular} almost complex structures, such that $\MM_{\bj}$ is a
smooth manifold for $\bj \in \JJ_{reg}$.  In particular, this means that there
exist regular $\bj$ arbitrarily close to $\mathbf{J^0}$.  The dimension of this
manifold is given by the Fredholm index, which in this case is $|x| - |y|$, the
difference of the Maslov indices of the intersections \cite{floerfredholm}.
Note also that $\MM_{\bj}(x,y)$ carries a free $\RR$-action by translation in
the $s$ variable and we shall denote the quotient space by
$\widehat{\MM}_{\bj}(x,y)$.

\subsection{Compactifications}
From this point onward we shall assume that $c_1(M)=0$.  This is independent of
the almost complex structure chosen.  From the previous section, we now know
that, given two intersection points $x$ and $y$, $\MM_{\bj}(x,y)$ is a smooth
manifold of the correct dimension, provided we pick $\bj \in
\mathcal{J}_{reg}$.  Given some real number $E$, we can restrict attention to
the set $\MM^E_{\bj}(x,y)$ of Floer trajectories with the energy bound $E(u)
<E$.  Gromov compactness says that this manifold admits a natural
compactification by adding broken trajectories, possibly with bubbles.  In
order to be able to define Floer cohomology, we shall need to look at the
compactifications of these moduli spaces in cases when they have dimension
$\leq 2$.

We want to prove that we can pick our almost complex structures $(J_t)$ in such
a way that we get no bubbling along solutions to
the Floer equation. There are two possible types of bubbles: discs appearing on
the Lagrangian boundary, and spheres appearing on the interior of some Floer
solution.  We shall prove that in the case where $c_1(M)=0$, we can exclude the
possibility of sphere bubbles.   Disc bubbles are more difficult and no general
approach exists to deal with these (in fact such an approach \emph{cannot}
exist in all situations as evidenced by the existence of \emph{obstructed}
Lagrangians \cite{fooo1}).  However, we shall show later that we can avoid such
bubbles in some specific cases.  To prove that we get no sphere bubbles, we
adapt the argument found in \cite{hofersalamon}.

Fix some nonzero homology class $A \in H_2(M;\ZZ)$.  For a given $J$, we can
consider the moduli space of simple $J$-holomorphic maps $v \colon S^2 \to M$
representing the homology class $A$, which we shall denote $\MM_s(A,J)$.  We
can also take a 1-parameter family $\bj=(J_t)$ and consider the space
\[
\MM_s(A,\bj) = \left\{(t,v) : v \in \MM_s(A,J_t)\right\}.
\]
We can also consider the universal moduli space
\[
\MM_s(A,\JJ) = \left\{(t,v,\bj) : (t,v) \in \MM_s(A,\bj) \right\}.
\]
This is a smooth Banach bundle and the projection to $\JJ$ is Fredholm of index
$2n + 2c_1(A)+1$, so that for $\bj \in \JJ'_{reg}$ some second category set of
almost complex strcutures, $\MM_s(A,\bj)$ is a smooth manifold of that
dimension.  The analysis underlying all this is similar to that in the previous
section and can be found, for example, in \cite{mcduffsalamonjholo}.  We also
note that $\MM_s(A,\bj)$ admits a free action by the real 6-dimensional
reparametrization group of the sphere $G = PSL(2,\CC)$ and we consider the
space $\MM_s(A,\bj) \times_G S^2$, which, for generic $\bj$, is a smooth
manifold of dimension $2n +2c_1(M) - 3$.

By taking the fibre product over $\JJ$, we can consider
\[
\mathcal{N} = \left( \MM_s(A,\JJ) \times_G S^2 \right) \times_{\JJ} \left(
\MM(x,y,\JJ) \times [0,1] \right)
\]
and the map
\[
\mathcal{N}  \to M \times [0,1] \times M \times [0,1]
\]
given by
\[
([v,z],t,u,t') \mapsto (v(z),t,u(0,t'),t').
\]
We want to know the intersection of the image of this map with the diagonal
$\Delta_{M \times [0,1]}$. Since $\MM_{\bj}(x,y)$ carries an $\RR$-action, if
there is any such intersection, there must be a bubble intersecting a Floer
solution $u$ at some $u(0,t)$, since we only care about $J_{t_0}$-bubbles
meeting some Floer solution at time $t_0$.

For any $t$, we have an evaluation map $ev_t \colon \MM_{\bj}(x,y) \to M$ given
by $ev_t(u) = u(0,t)$ and a version of Proposition 3.4.2 in
\cite{mcduffsalamonjholo} says that this map is a submersion for all $t$.  This
means that the intersection with the diagonal is transverse, and therefore the
space
\[
\mathcal{Z} = \left\{ ([v,z],t,u,t')  : (v(z),t)=(u(0,t'),t')\right\}
\]
is a submanifold of $(\MM_s(A,\JJ) \times_G S^2) \times_{\JJ} (\MM(x,y,\JJ) \times
[0,1]) $ of codimension $2n+1$.  This means that the projection $\mathcal{Z} \to
\JJ$ has Fredholm index
\begin{align*}
&(2n + 2c_1(A) - 3) + (|x| - |y|+1) - (2n+1) \\
= \; &2c_1(A) + |x|-|y| -3 .
\end{align*}

Since we have $c_1=0$, this means that for generic $\bj=(J_t)$, the 1- and
2-dimensional moduli spaces of Floer solutions (which are needed to define the
Floer differential $d$ and show that $d^2=0$) will not intersect any sphere
bubbles.  Bearing in mind that the compactification of these spaces involves
adding broken solutions, possibly with bubbles, the same argument as in
\cite{hofersalamon} shows that we still get no intersection even after
compactifying our spaces.

The case of disc bubbles is more difficult and there is no general approach
that will work, but if we have chosen appropriate $J_0,J_1$ such that we get no
disc bubbles for our Lagrangians, then picking a generic path of almost complex
structures $(J_t)$ interpolating between these two gives a family of $(J_t)$
such that we can in fact define $HF(L_0,L_1)$.  This will be discussed more in
Section \ref{sec:discs}.

\subsection{Floer cohomology}
\label{sec:HFdefn}
We first fix the coefficient field we shall use.  Although (subject to certain
topological assumptions) the relevant moduli spaces can be oriented so that
Floer cohomology can be defined over fields of arbitrary characteristic, we
don't need this for our purposes.  We therefore introduce the Novikov
ring
\[
\Lambda_{\ZZ/2} = \left\{  \sum_r a_r q^r : a_r \in \ZZ/2, r \in \RR, r \to
\infty, \# \{a_r \neq 0 : r<E\}  < \infty \text{ for all }E  \right\}
\]
of power series in the formal parameter $q$ as in the Introduction.  This is in fact a field.

In order to define Floer cohomology, we define the Floer cochain complex to be
\[
CF(L_0,L_1) = \bigoplus_{x \in L_0 \cap L_1} \Lambda_{\ZZ/2}\langle x \rangle.
\]
In the case where $|y| = |x|-1$, the Floer differential is defined by
\[
dy = \sum_{u \in \widehat{\MM}_{\bj}(x,y)} q^{E(u)}x.
\]
For this map to be well-defined over the Novikov ring, for any $E$, there must
be only finitely many terms involving powers of $q$ less than $E$.  This
follows from Gromov compactness.  When $|y| = |x|-1$, the compactification of
$\widehat{\MM}_{\bj}^E(x,y)$ can only involve adding bubbles, since breaking
cannot occur as the solutions are already of minimal index.  But we have shown
that we can pick $\bj$ such that no bubbling occurs.  Therefore the
zero-dimensional manifold $\widehat{\MM}^E_{\bj}(x,y)$ is compact, hence
consists of finitely many points.

In order to show that this is in fact a differential (i.e. that $d^2=0$), the
standard approach is to identify the boundary of the compactification of any
1-dimensional $\widehat{\MM}_{\bj}(x,z)$ with $\widehat{\MM}_{\bj}(x,y) \times
\widehat{\MM}_{\bj}(y,z)$, and use the fact that boundary points of a
1-manifold come in pairs.  This identification again relies on the fact that no bubbing occurs, which is ensured by the previous section.  Once again we stress that we have not yet dealt with disc bubbling, so that the content of this section is incomplete and Floer cohomology will not be properly defined until we do so in Section \ref{sec:discs}.

In our setting, where $c_1(M)=0$, we may also pick a grading so that
$HF^*(L_0,L_1)$ becomes a $\ZZ$-graded group \cite{seidelgradings}.

We also want to define a multiplication map on Floer cohomology.  We start by doing this on the chain level.

Consider three Lagrangian submanifolds $L_i$, $i=0,1,2$ and transverse intersection points $x \in L_0 \cap L_2$, $y \in L_0\cap L_1$, $z \in L_1\cap L_2$.  Similar to before we may consider the moduli space $\MM^2_{\bj}(x,y,z)$ of holomorphic curves $u$ from a disc with 3 marked boundary points mapping to $M$ such that the marked boundary points tend to $x,y,z$ and the remainder of the boundary maps to the various Lagrangians (see \cite[Section 2]{seidelbible} for more specific details).  Here $\bj$ is a 2-parameter family of almost complex structures $(J_w)_{w \in \DD}$ and a similar analysis to the previous section shows that, for a generic choice of $\bj$, $\MM^2_{\bj}(x,y,z)$ is a smooth manifold of dimension $|x| -|y| - |z|$.

We can therefore define
\begin{align*}
&m \colon CF(L_1,L_2) \otimes CF(L_0,L_1) \to CF(L_0,L_2), \\
&m(z,y) = \sum_{u \in \MM^2_{\bj}(x,y,z)} q^{E(u)}x.
\end{align*}
in the case where $|x|=|y| + |z|$.  We want this to be a chain map so that we get a multiplication on the cohomological level.

Here the standard approach is again to consider the boundary of the compactification of the 1-dimensional part of $\MM^2_{\bj}(x,y,z)$ (see for example \cite{rittertqft}).  However, in our case we must once more rule out the possibility of bubbling off of spheres (disc bubbles will be dealt with in Section \ref{sec:discs}).

We continue in a similar vein to before and consider the universal moduli space
\[
\MM^2(x,y,z,\JJ)= \left\{ (u,\bj): u \in \MM^2_{\bj}(x,y,z) \right\}
\]
for an appropriate Banach space $\JJ$ of 2-parameter families of almost complex structures defined similarly to the previous section.  We then consider
\[
\mathcal{N}' = \left( \MM_s(A,\JJ) \times_G S^2 \right) \times_{\JJ} \left(
\MM^2(x,y,z,\JJ) \times \DD \right).
\]
By mapping to $M\times \DD \times M \times \DD$ via $([v,z],w,u,w') \mapsto (v(z),w,u(w'),w')$, we see that $\mathcal{N}'$ contains a submanifold 
\[
\mathcal{Z}' = \left\{ ([v,z],w,u,w')  : (v(z),w)=(u(w'),w')\right\}
\]
of codimension $2n+2$, which represents the intersections between $J_{w}$-bubbles and multiplication curves $u$ at point $u(w)$.  The projection $\mathcal{Z}' \to \JJ$ is Fredholm of index
\begin{align}
\label{eqn:dim}
&(2n + 2c_1(A) - 2) + (|x| - |y| -|z| +2) - (2n+2) \\   \nonumber
= \; &2c_1(A) + |x|-|y|-|z| -2 .
\end{align}
Therefore, for generic $\bj=(J_w)$, the 0- and 1-dimensional moduli spaces of such holomorphic discs do not intersect any sphere bubbles (recall that we are assuming $c_1(M)=0$), so these will not obstruct our multiplication surviving to cohomology.

We shall also want, when defining \emph{wrapped} Floer cohomology, to have a map
\[
\Psi_H \colon CF(L_0,L_1) \to CF(L_0, \psi_H(L_1)),
\]
where $\psi_H$ is the Hamiltonian isotopy coming from some Hamiltonian $H \colon M \times [0,1] \to \RR$ (when $M$ is noncompact but convex at infinity, we additionally require $H$ to be \emph{monotone}: $\partial_sH_s \leq 0$ \cite{rittertqft}).  First note that intersection points $y \in L_0 \cap \psi(L_1)$ are in one-to-one correspondence with \emph{Hamiltonian chords} $y \colon [0,1] \to M$ such that $y(0) \in L_0, y(1) \in L_1$, and $\dot{y}(s) = X_H(y(s))$.

For $x \in L_0 \cap L_1$ and $y \in L_0 \cap \psi(L_1)$, we consider the moduli space of \emph{continuation Floer trajectories} $\MM_{\bj}^H(x,y)$, solutions $u$ to the equation
\[
\partial_s v+ J_{s,t}(\partial_t v-X_H) = 0
\]
on the strip $\RR \times [0,1]$ such that $u(\cdot,0) \in L_0$, $u(\cdot,1) \in L_1$, and which converge to the point $x$ at $+\infty$ and to the chord $y(t)$ at $-\infty$. The standard approach \cite{abouzaidseidelviterbo} shows that, for generic $\bj = (J_{s,t})$, this moduli space is a smooth manifold of dimension $|y|-|x|$ and we can define
\[
\Psi_H x = \sum_{u \in \MM^H_{\bj}(x,y)} q^{E(u)}y
\]
in the case when $|y|=|x|$.  Again the standard argument involving the 1-dimensional part of $\MM^H_{\bj}(x,y)$ shows that this is a chain map modulo bubbling.  But no bubbling of spheres occurs because of the same dimension count as in \eqref{eqn:dim} replacing $\vdim \MM_{\bj}^2(x,y,z)$ with $\vdim \MM_{\bj}^H(x,y)$: the space $\mathcal{Z}''$ representing intersections between $J_{s,t}$-bubbles and continuation trajectories at $u(s,t)$ has virtual dimension
\begin{align*}
&(2n + 2c_1(A) - 2) + (|y| - |x|+2) - (2n+2) \\ 
=\; &2c_1(A) + |y|-|x| -2 .
\end{align*}

Note that we are here using 2-parameter families of almost complex structures on $\RR \times [0,1]$ as opposed to the 1-parameter families used in defining $d$.  See Section \ref{sec:discs} for the argument for disc bubbles.

A similar argument shows that $\Psi_H$ intertwines the multiplicative structures on $HF(L_0,L_1)$ and $HF(L_0, \psi_H(L_1))$.

\begin{rmk}
In the case of exact Lagrangians in an exact symplectic manifold, much of the
above analysis is unnecessary: exactness means that no bubbles occur in the
compactifications of our moduli spaces, and we also get a priori energy bounds
independent of $u$, so we can actually work over $\ZZ/2$ should we wish.
\end{rmk}

\subsection{Floer cohomology in Lefschetz fibrations}
\label{sec:acs}
In the context of a Lefschetz fibration $\pi \colon E \to \CC$, we can make a
choice of almost complex structures which lends itself well to Floer
cohomology calculations.

In some neighbourhood of $E^{crit}$ we pick $J$ to agree with the standard
integrable complex structure in the local model $z \mapsto \sum z^2_i$ as in
Definition \ref{defn:lef}, which makes $\o$ locally a K\"{a}hler form.  Away
from $E^{crit}$, we have the splitting
\[
T_xE = T^h_xE \oplus T^v_x E
\]
where $T^v_xE = \ker(D\pi_x)$ and $T^h_x E \cong T_{\pi(x)}\CC$.  With respect
to this splitting, we choose $J$ that, away from $E^{crit}$, look like
\[
\left( \begin{array}{cc}
j &  0 \\
0 & J^v \end{array}
\right),
\]
such that $J^v$, the vertical part of $J$, is compatible with $\o$ restricted
to the fibre and $j$ is compatible with the standard form on the base.  Such a
$J$ makes the projection $\pi$ $J$-holomorphic, so that Floer solutions in $E$
project to $j$-holomorphic strips $\pi \circ u \colon \Sigma \to \CC$, and we
can now use the maximum principle for holomorphic functions to restrict the
region in which Floer solutions may appear.

The problem is that such a $J$ will not necessarily be regular, so not be
suitable for defining $HF(L_0,L_1)$. In \cite{maydanskiyseidel}, they proceed
as follows.  They take some small generic perturbation of $(J_t)$ to regular
$(\tilde{J}_t)$ such that $(\tilde{J}_t)$ is still $\o$-compatible, losing in
the process the property that $\pi$ is holomorphic.  However, Gromov
compactness says that Floer solutions for $(J_t)$ will be close to Floer
solutions for $(\tilde{J}_t)$.  In order to apply Gromov's compactness theorem
for this argument to work, we need some energy bounds, which a priori exist in
the setting of \cite{maydanskiyseidel} as all their manifolds are exact.

We do not have any such energy bounds.  Therefore, we perturb $J$ by adding
some \emph{horizontal} component to get
\[
\tilde{J}=
\left( \begin{array}{cc}
j &  0 \\
H & J^v \end{array}
\right),
\]
where $H$ is some small perturbation that is zero on some neighbourhood of the
intersctions of our Lagrangians and such that $\tilde{J}^2=-1$.  Now $\pi$ is
still holomorphic, so we can use maximum principles in the base, but
$\tilde{J}$ is no longer compatible with $\o$.  However, for small $H$, it will
still tame $\o$ and we can use the discussion above to say that we can still do
Floer cohomology in this setting.  The proof that the space of such $H$ is large enough for us to achieve transversality as in Section \ref{sec:fredholm} can be found in \cite[Lemma 2.4]{seidelles}.

\subsection{Disc bubbles}
\label{sec:discs}
We have not yet said anything about how to avoid disc bubbles, $J$-holomorphic
maps $w \colon (\DD, \partial \DD) \to (M,L)$.  However, for the purposes of
this paper, we need only consider specific sorts of Lagrangian submanifolds,
namely spheres or Lefschetz thimbles in some Lefschetz fibration, with a
six-dimensional total space and whose first Chern class vanishes.

In the first instance, it is shown in \cite[Corollary 4.5]{welschinger}, using techniques
inspired by symplectic field theory, that for a Lagrangian sphere $L$ in a
symplectic manifold of dimension at least 4 with vanishing first Chern class,
there exists a $J_L$ such that the Floer cohomology of $L$ is unobstructed ($(L,J_L)$ is an \emph{elementary Lagrangian conductor} in the language of Welschinger) and moreover we have the classical isomorphism $HF^*(L,L) \cong H^*(S^n, \Lambda_{\ZZ/2})$ \cite[Corollary 4.12]{welschinger}.  This is proven in \cite{welschinger} only for compatible $J$, not the larger class of almost complex structures we have considered in this section.  However, in the next section, there is only one point at which we need to consider the Floer cohomology of a 3-sphere in the total space of a Lefschetz fibration (Section \ref{sec:wrapped}) and here we don't need to perform the horizontal perturbation trick, so at this point in the argument we can just pick a compatible $J$ for the sphere as usual.

As for thimbles, we start by picking $J$ adapted to our Lefschetz fibration as
above.  If a disc bubble exists, then by considering the projection to the
base, we see that any such bubble must entirely be contained in some fibre of
$\pi \colon E \to \CC$.  The part of the thimble living in this fibre is now
just a sphere, so we can arrange for the vertical part $J^v$ of $J$ to be such
that we get no bubbles as in the previous paragraph.  However, this fails to
take into account of the fact that we have a 1-parameter family of such fibres
(the vanishing path).  In fact, in \cite{welschinger} the relevant Fredholm
problem involves a Fredholm operator whose index is bounded from above by $-2$,
so we may in fact generically pick a 1-parameter family of such $J$ so that the Floer cohomology is unobstructed.

Now to complete the definition of the Floer cohomology of two such Lagrangians, we pick appropriate
$J_0$ and $J_1$ as above and then pick some path $\bj =(J_t)$ interpolating
between them.  A generic perturbation of $\bj$, which may be chosen such that
the endpoints are fixed will then be suitable.  We may do likewise to exclude the possibility of disc bubbles appearing in the compactifications of $\MM_{\bj}^2(x,y,z)$ and $\MM_{\bj}^H(x,y)$ (although we now consider 2-parameter families of almost complex structures, we are free to choose that $\bj$ be constant along the boundary components of the disc/strip since we can achieve transversality by perturbing $\bj$ just on the interior), thus completing the constructions of Section \ref{sec:HFdefn}.

\begin{rmk}
Welschinger \cite{welschinger} establishes a result saying that, given a Lagrangian sphere $L$ and any $E>0$, there exists a second category set of almost complex structures $J_E$ such that $E(w)>E$ for any $J_E$-holomorphic disc $w \colon (\DD, \partial \DD) \to (M,L)$.  These bubbles can then be discounted by using the Novikov ring.  A similar sort of argument is perhaps best explained in \cite[Chapter 4.6]{fooo1}.

Briefly, suppose that we pick almost complex structures $J_i$ such that any $J_i$-holomorphic disc $w$ has energy $E(w) >i$. We can construct, for each $i$, $\ainf$-structures $\{\mu^i_d\}$ on the space of cochains $C^*(L)$ where, by assumption, $\{\mu^i_d\}=1 + O(q^i)$.  (Here $q$ is our formal Novikov parameter.)  In \cite[Chapter 4.6]{fooo1}, they construct $\ainf$-functors $\mathcal{F}^i \colon (C^*(L), \{\mu^i_d\}) \to (C^*(L), \{\mu^{i+1}_d\})$ which come from counts of genus 0 stable curves, all of whose components are $J_{\alpha}$-holomorphic for some $i \leq \alpha \leq i+1$.  Again, $\mathcal{F}^i=\{\mathcal{F}^i_r\} = \Id + O(q^i)$ in our Novikov filtration.  This means that $\prod_{i=1}^{\infty} \mathcal{F}^i$ converges over $\Lambda_{\ZZ/2}$, and so defines an $\ainf$-functor from $(C^*(L), \{\mu^1_d\})$ to the classical $\ainf$-structure on $C^*(L)$.  We may then pull back the classical Maurer-Cartan solution for which $HF^*(L,L) \cong H^*(S^n, \Lambda_{\ZZ/2})$ by $\prod_{i=1}^{\infty} \mathcal{F}^i$.
\end{rmk}

\section{The examples of Maydanskiy-Seidel}
\label{sec:maydanskiyseidel}
Using the same method as explained in Section \ref{sec:lef}, we can construct
the six-dimensional symplectic manifold $X_2$ in Figure
\ref{fig:maydanskiyseidel}.  Its generic fibre is diffeomorphic to the
$A_{m+1}$ Milnor fibre $M_{m+1}$ and the Lefschetz fibration $\pi \colon X_2
\to \CC$ has $m+1$ critical points corresponding to $m+1$ vanishing cycles in
$M_{m+1}$.  The first $m$, $V_1, \ldots, V_m$ come from the straightline
matching paths, but $V_{m+1}$ is the sphere associated to the curved path
$\gamma_{m+1}$.  For each critical value $x_i$, corresponding to $V_i$, fix
some vanishing path $\beta_i \colon [0,\infty) \to \CC$ such that $\beta_i(t) =
t$ for $t \gg 0$.  Let $\Delta_i \subset X_2$ denote the corresponding
Lefschetz thimble.

\begin{figure}[h!]
\[
\xy
(-20,0)*{\times}; (-10,0)*{\times}; **\crv{};
(-10,0); (0,0)*{\times}; **\crv{};
(10,0)*{\times};
(0,0); (1.5,0); **\crv{};
(5,0)*{\cdots};
(8.5,0); (10,0); **\crv{};
(30,0)*{\times};
(18,10); (18,-10); **\dir{.};
(19,10); (19,-10); **\dir{.};
(20,10); (20,-10); **\dir{.};
(21,10); (21,-10); **\dir{.};
(22,10); (22,-10); **\dir{.};
(18.5,10); (18.5,-10); **\dir{.};
(19.5,10); (19.5,-10); **\dir{.};
(20.5,10); (20.5,-10); **\dir{.};
(21.5,10); (21.5,-10); **\dir{.};
(20,-13)*{R};
(-20,0); (10,0); **\crv{(-10,-18)&(1,20)};
(-5,7)*{\gamma_{m+1}};
(-30,0)*{X_2=};
\endxy
\]
\caption{}
\label{fig:maydanskiyseidel}
\end{figure}

A trivial extension of the argument in \cite{maydanskiyseidel}, which will be recapped in this section, shows that
$X_2$ is diffeomorphic to $T^*S^3 \cup 2$-handle and also contains no Lagrangian sphere $L$ such that $[L]\neq0$ in
$H_2(X_2;\ZZ/2)$.  (We have shown below only one such possible choice of
$\gamma_{m+1}$; there are infinitely many others for which this is also true
\cite{maydanskiyseidel}.)  We construct a deformation $\tilde{X}_2$ of this
manifold by adding on a closed 2-form supported in the shaded region $R$, as in
Section \ref{sec:maydanskiy}, to obtain a family of symplectic manifolds
$(\tilde{X}^t_2, \o_t)$. $c_1(X_2)=0$ so therefore $c_1(\tilde{X}^t_2)=0$ for
all $t$.  We also note that, after deformation, the $V_i$ will still be
Lagrangian in $M_{m+1}$ since they live away from the region $R$.  Also the
thimbles $\Delta_i$ will stay Lagrangian in $\tilde{X}^t_2$.

In this section, we shall prove the following:

\begin{thm}
\label{thm:maydanskiyseidel}
For all $t \in [0,1]$, $\tilde{X}^t_2$ contains no Lagrangian sphere $L$ such
that $[L] \neq 0 \in H_2(\tilde{X}^t_2; \ZZ/2)$.
\end{thm}

The proof of this will essentially just be a repeat of the argument in
\cite{maydanskiyseidel}, so we shall not explain all the details fully, instead
directing the interested reader to the relevant sections of
\cite{maydanskiyseidel}.  However, this proof relies heavily on the technology
of Floer cohomology and Fukaya categories.  In the original paper,
everything is carried out working within the category of exact symplectic
manifolds so the analytical issues involved in setting up Floer cohomology are
easily overcome.  This was why we had to go through the analysis of the
previous section as we now often have to work in the more problematic nonexact
setting.  With the results of the previous section however, the argument of
\cite{maydanskiyseidel} more or less just carries over, and we only make a few
remarks where particular care needs to be exercised.

In what follows, we shall denote by $HF_t^*(L_0,L_1)$ the Floer cohomology
computed with respect to $\o_t$ in any situations where there is likely to be
confusion about the symplectic form being used.

\subsection{Wrapped Floer cohomology}
\label{sec:wrapped}
We start by defining a variant of Floer cohomology, wrapped Floer cohomology.
Following \cite{maydanskiyseidel}, we shall not need to define this in the
level of generality found in \cite{abouzaidseidelviterbo,rittertqft}, but
instead restrict to a simpler (and, in our setting, equivalent) definition
which is well-suited to Lefschetz fibrations.

Given a Lefschetz fibration $\pi \colon E \to \CC$, we consider a Hamiltonian
$H \colon E \to \RR$ of the form $H(y) = \psi(\frac{1}{2}|\pi(y)|^2)$ where
$\psi \colon \RR \to \RR$ is such that $\psi(r) = 0$ for $r < 1/2$ and
$\psi'(r)=1$ for $r \gg 0$.  Let $\Phi^{\alpha}$ denote the time-$\alpha$ flow
of this Hamiltonian and, given some Lagrangian $L$, we define $L^{\alpha} =
\Phi^{\alpha}(L)$.

We can now define the wrapped Floer cohomology of a Lagrangian $L$ and a
thimble $\Delta$ (where, in order to exclude bubbling of discs as mentioned
previously, $L$ is either a sphere or another thimble) to be the direct limit
of Floer cohomology groups
\[
HW_t^*(L,\Delta) = \varinjlim_k HF_t^*(L,\Delta^{2\pi k + \epsilon})
\]
for some very small $\epsilon>0$.  The maps involved in this direct limit are the continuation maps from Section \ref{sec:HFdefn}.

We will actually need to perform an extra small Hamiltonian isotopy in addition
to $\Phi^{\alpha}$ in order to ensure transversality of intersections but will
suppress further mention of this.  For our purposes, it is not necessary to
identify our Floer groups canonically so the details of how we do this are
irrelevant for what follows.
%

To prove Theorem \ref{thm:maydanskiyseidel}, suppose for sake of
contradiction that there does exist a Lagrangian sphere $L \subset
\tilde{X}^t_2$ such that $[L] \neq 0$ in $H_*(\tilde{X}^t_2;\ZZ/2)$.
$\tilde{X}^t_2$ is topologically $T^*S^3$ with a 2-handle attached, and it is
shown in \cite[Section 9]{maydanskiyseidel} that $L \cdot \Delta_{m+1} \neq 0$
for such a sphere .  This intersection number is the Euler characteristic of
the Floer cohomology group $HF_t^*(L,\Delta_{m+1})$.  Given the compactness of
$L$, this group is equal to the wrapped Floer cohomology group $HW_t^*(L,\Delta_{m+1})$ (we may choose to start ``wrapping'' outside some compact set containing $L$) and $HW_t^*(L,\Delta_{m+1})$ is itself a module over the unital ring $HW_t^*(\Delta_{m+1},\Delta_{m+1})$, where the multiplication maps here are the images under the direct limit of the multiplication defined in Section \ref{sec:HFdefn}. Thus we
conclude

\begin{lem}
\label{thm:hwvanishes}
If such a Lagrangian sphere exists, then $HW_t^*(\Delta_{m+1},\Delta_{m+1})
\neq 0$.
\end{lem}

The rest of this section is devoted to proving that
$HW_t^*(\Delta_{m+1},\Delta_{m+1}) = 0$ to provide the required contradiction.

\subsection{From total space to fibre}
If we consider the directed system of groups used to define
$HW_t^*(\Delta_{m+1}, \Delta_{m+1})$, we see that each step introduces new
intersection points as the path over which our wrapped Lefschetz thimble lives
wraps round the base once more.  Choose our family of almost complex structures
$(J_t)$ as in Section \ref{sec:acs}.  In \cite{maydanskiyseidel}, they
establish the existence of a spectral sequence computing the wrapped Floer
cohomology of any two thimbles, which carries over in our setting in light of
the discussion of Section \ref{sec:floer}.  When we consider
$HW_t^*(\Delta_{m+1}, \Delta_{m+1})$, this spectral sequence yields the
following long exact sequence
\[
\xymatrix{
{HF^*_t(\Delta_{m+1},\Delta^{\epsilon}_{m+1})} \ar[r] &
{HF^*_t(\Delta_{m+1},\Delta_{m+1}^{2\pi+ \epsilon})} \ar[d]  \\
 & {HF^*_t(\mu(V_{m+1}),V_{m+1}),} \ar[ul]^{\sigma}
}
\]
where the bottom group is calculated in the fibre $E_z$ and $\mu$ denotes the
outer monodromy of the Lefschetz fibration. Lemma \ref{thm:build} allows us to
identify some particular fibre $E_{z'}$ with the manifold $M$ included in the
data of this lemma.  We may arrange that $z=z'$.

In particular, since the unit in $HW_t^*(\Delta_{m+1},\Delta_{m+1})$ arises as
the image of $1 \in HF_t^*(\Delta_{m+1},\Delta^{\epsilon}_{m+1}) =
\Lambda_{\ZZ/2}$, the map $\sigma$ must be zero.  By analysing the curves
involved in defining the map $\sigma$ \cite[Section 5]{maydanskiyseidel} and
comparing to the maps involved in Seidel's the long exact sequence
\cite{seidelles}, we can, by Poincar\'{e} duality, identify the map $\sigma$
with an element of $HF_t^0(V_{m+1},\mu(V_{m+1}))$, which we shall also denote
by $\sigma$.

\begin{lem}
\emph{(\cite[Proposition 5.1]{maydanskiyseidel})}
If $HW_t^*(\Delta_{m+1},\Delta_{m+1}) \neq 0$, then $\sigma$ vanishes.
\end{lem}

\subsection{Fukaya categories}
\label{sec:fukaya}
We now shift attention to the Fukaya category of the fibre $\FF(E_z)$, and
introduce two related categories.

The first is a directed $\ainf$-category $\AA$, which has as objects the finite
set $\{V_1, \ldots, V_m\}$ and morphisms
\[
hom_{\mathcal{A}}(V_i,V_j) =
\begin{cases}
(\ZZ/2)e_i & \text{for $i=j$} \\
(\ZZ/2)f_i & \text{for $i=j-1$} \\
0 & \text{otherwise,}
\end{cases}
\]
where the degrees are chosen to be $|e_i|=0$ and $|f_i|=1$.  This category is
chosen to reflect the fact that we have an $A_m$ configuration of Lagrangian
spheres $V_i \subset M_m$ coming from the straightline paths in Figure
\ref{fig:maydanskiyseidel}, where the only points of intersection are between
adjacent spheres and the gradings can be chosen in a nice way.  This determines
the higher-order $\ainf$-structure, namely that the only nontrivial higher
products are given by $\mu^2(e_i,e_i)=e_i$ and $\mu^2(f_i,e_i)=f_i
=\mu^2(e_{i+1},f_i)$.

The second variant of the Fukaya category we shall consider is the
$\ainf$-category $\BB$, which is the subcategory of the Fukaya category
$\FF(E_z)$ generated by the following collection of Lagrangian submanifolds
\[
V_1, \ldots, V_m,V_{m+1}, \tau_{V_m}(V_{m+1}),
\tau_{V_{m-1}}\tau_{V_m}(V_{m+1}), \ldots, \tau_{V_1}\ldots
\tau_{V_m}(V_{m+1}).
\]
In \cite{maydanskiyseidel}, there is no need to restrict attention specifically
to $\BB$ and we can happily work with the whole Fukaya category $\FF(E_z)$,
even though as above we do not strictly need to.  However, all the objects in
$\BB$ are disjoint from the region $R$ where $\o_t$ is nonexact and we can use
maximum principles to ensure that all pseudoholomorphic curves between these
objects also do not enter the region $R$.  This means there is no extra
analysis to do in defining the $\ainf$-category $\BB$ as we are essentially
just in an exact setting.

In what follows, we shall also want to use Seidel's long exact sequence in
Floer cohomology \cite{seidelles}.  Part of the proof of this long exact
sequence in \cite{seidelles} relies on a spectral sequence argument coming from
a filtration on Floer cochain groups given by the symplectic action
functional.  Seidel needs to upgrade this $\RR$-filtration to some
$\ZZ$-subfiltration in order to show that a certain mapping cone is acyclic, which can be done since the action spectrum will be discrete for finitely many exact Lagrangians in an exact symplectic manifold.  In $\BB$, this argument remains valid since maximum principles mean that we are considering the same holomorphic curves with the same actions as in the exact case, although this approach would not work in general.

We can consider the ``derived'' versions of $\AA$ and $\BB$ defined via twisted
complexes as $D\AA = H^0(Tw\AA)$ and $D\BB = H^0(Tw\BB)$ \cite{seidelbible}.
There is a canonical (up to quasi-isomophism) functor $\iota : \AA \to \BB$
which on the derived level extends to an exact functor $D\iota : D\AA \to
D\BB$.

On the level of derived Fukaya categories $D\BB$, thanks to the result of
Seidel \cite{seidelbible} relating algebraic and geometric twisting operations,
$\sigma$ corresponds to an element $S \in hom_{D\BB}(V_{m+1}, T_{V_1} \cdots
T_{V_m} V_{m+1})$.  If $\sigma$ vanishes $S$ must too, so, looking at exact
triangles in $D\BB$, this means that
\[
V_{m+1}[1] \oplus T_{V_{1}} \cdots T_{V_{m}} V_{m+1} \cong Cone(S),
\]
so we wish to understand $C=Cone(S)$.

Given all this, the next lemma is pure algebra.

\begin{lem}
\emph{(\cite[Proposition 6.2]{maydanskiyseidel})}
If $S=0$, then $V_{m+1}$ is isomorphic to a direct summand of an object lying
in the image of the functor $D\iota \colon D\AA \to D\BB$.
\end{lem}

\subsection{Contradiction}
The fibre $E_z$ itself admits a Lefschetz fibration as pictured at the start of
this section, such that the matching cycles of interest arise from matching
paths $\gamma_1, \ldots, \gamma_{m+1}$.  By assumption, $\gamma_{m+1}$ is not
isotopic to $\gamma_i$ for $1\leq i \leq m$ within the class of paths which
avoid the critical values except possibly at their endpoints.

\begin{lem}
\label{thm:contradiction}
\emph{(\cite[Lemma 7.2]{maydanskiyseidel})}
For $1 \leq i \leq m$, and for all $t \in [0,1]$, the image of the product map
\[
HF_t^*(V_{m+1},V_i) \otimes HF_t^*(V_i,V_{m+1}) \to HF_t^*(V_{m+1},V_{m+1})
\cong H^*(V_{m+1}; \Lambda_{\ZZ/2})
\]
does not contain the identity in its image.
\end{lem}

As in \cite{maydanskiyseidel}, this is proved by considering the auxiliary Lagrangian $L_{\xi} \cong S^{1}\times \RR$ associated to the path $\xi$ in Figure \ref{fig:contradiction}. The key point is that, since $\gamma_i$ is not isotopic to
$\gamma_{m+1}$, we can draw $\xi$ so that it intersects $\gamma_{m+1}$ but is
disjoint from $\gamma_i$ (here we
have drawn only two of the matching paths, $\gamma_{m+1}$ and $\gamma_i$, to
avoid clutter).  

It is proven in \cite{khovanovseidel} that $\dim HF_t^*(L_{\xi}, V_{m+1}) >0$,
whereas clearly we have $\dim HF_t^*(L_{\xi},V_i)=0$.  As before, we may choose $\xi$
to lie away from the region $R$ where our deforming 2-form is supported since, by assumption, this also true for the paths $\gamma_j$, so once
more we may use maximum principles to restrict all Floer solutions to a region
of $M_{m+1}$ where $\o_t$ is exact.

\begin{figure}[h!]
\[
\xy
(-20,0)*{\times};
(-10,0)*{\times};
(-10,0); (0,0)*{\times}; **\crv{};
(10,0)*{\times};
(5,0)*{\cdots};
(30,0)*{\times};
(18,10); (18,-10); **\dir{.};
(19,10); (19,-10); **\dir{.};
(20,10); (20,-10); **\dir{.};
(21,10); (21,-10); **\dir{.};
(22,10); (22,-10); **\dir{.};
(18.5,10); (18.5,-10); **\dir{.};
(19.5,10); (19.5,-10); **\dir{.};
(20.5,10); (20.5,-10); **\dir{.};
(21.5,10); (21.5,-10); **\dir{.};
(20,-13)*{R};
(-20,0); (10,0); **\crv{(-10,-18)&(1,20)};
(-7,-2)*{\gamma_i};
(-5,7)*{\gamma_{m+1}};
(5,0); (-25,15); **\crv{(0,15)&(5,15)};
(5,0); (-25,-15); **\crv{(0,-14)&(6,-16)};
(-10,-11.5)*{\xi};
\endxy
\]
\caption{}
\label{fig:contradiction}
\end{figure}

Suppose we have elements $a_1 \in HF_t^*(V_{m+1},V_i)$ and $a_2 \in
HF_t^*(V_i,V_{m+1})$ such that $a_2 \cdot a_1 \in H^0(V_{m+1})$, the invertible
part of this ring.

This then means that the composition
\[
HF_t^*(L_{\xi},V_{m+1})\stackrel{a_1 \cdot}{\rightarrow} HF_t^*(L_{\xi},V_i)
\stackrel{a_2 \cdot}{\rightarrow}
HF_t^*(L_{\xi},V_{m+1})
\]
is an isomorphism, which is a contradiction.

Once we have this, we can complete the proof of Theorem
\ref{thm:maydanskiyseidel}, the remainder of which carries over directly from
\cite{maydanskiyseidel} as it is essentially just algebra.

Suppose that $HW_t^*(\Delta_{m+1},\Delta_{m+1}) \neq 0$.  Then $V_{m+1}$ is
contained in the image of $D\iota: D\AA \to D\BB$.  Say that $V_{m+1}$ occurs
as a direct summand of $C$ in the image.  Then, in particular
\[
hom_{D\BB}(C,V_{m+1}) \otimes hom_{D\BB}(V_{m+1},C) \to
hom_{D\BB}(V_{m+1},V_{m+1}) \cong H^*(S^n;\Lambda_{\ZZ/2})
\]
contains the identity in its image as we can consider the maps corresponding to
projection and inclusion with respect to this summand.  However, thanks to the
particularly simple form of $\AA$, there exists a classification of twisted
complexes in $\AA$, following from Gabriel's theorem \cite{gabriel}. It says
that any twisted complex is isomorphic to a direct sum of (possibly shifted
copies of) the basic complexes $C_{kl}$
\[
C_{kl} =
\begin{cases}
W_i = \ZZ/2 & \text{for $k\leq i<l$ concentrated in degree $0$} \\
W_i=0 & \text{otherwise} \\
\delta_{i+1,i} = f_i  & \text{for $k\leq i<l$} \\
\delta_{ij}=0 & \text{otherwise.}
\end{cases}
\]
However, by repeated application of our Lemma \ref{thm:contradiction} above, we
derive a contradiction, since the terms in the $C_{kl}$ correspond
geometrically to $V_i$ involved there.  This completes the proof that
$HW_t^*(\Delta_{m+1},\Delta_{m+1})=0$, and therefore, by Lemma
\ref{thm:hwvanishes}, there cannot exist a homologically essential Lagrangian
sphere in $\tilde{X}_2^t$.

\section{\texorpdfstring{Distinguishing $X_1$ and $X_2$}{Distinguishing X1 and
X2}}
\label{sec:end}

\subsection{Moser for symplectic manifolds convex at infinity}
Take a symplectic manifold $\mo$ which is convex at infinity.  Recall that this
means that there is a relatively compact set $M^{in}$ such that on a
neighbourhood of the boundary $\partial M^{in}$ we have a 1-form $\theta$ such
that $d\theta = \o$ and $\theta|_{\partial M^{in}}$ is a contact 1-form, and
that $M \setminus M^{in}$ looks like the positive symplectization of $\partial
M^{in}$ according to $\theta|_{\partial M^{in}}$.

Suppose that we have a family of cohomologous 2-forms $(\o_t)_{t \in [0,1]}$
which make $M^{in}$ a symplectic manifold with convex boundary.  We can
complete $(M^{in},\o_t)$ to a family $(M,\widehat{\o}_t)$ of noncompact
symplectic manifolds with cohomologous symplectic forms all convex at infinity.
We want to prove a version of Moser's theorem \cite{moser} in this setting.

\begin{lem}
\label{thm:moser}
The family $(M,{\o}_t)$ above are all symplectomorphic, by symplectomorphisms
modelled on contactomorphisms at infinity.
\end{lem}

\begin{proof}
We follow the standard argument, but need to pay attention to possible problems
arising from the noncompactness of $M$.  Since the $\o_t$ are all cohomologous,
we pick $\sigma_t$ such that
\[
 \frac{d}{dt} \o_t = d \sigma_t.
\]
Then, Moser's theorem follows from considering the flow ${\psi_t}$ defined by
integrating the vector fields $Y_t$ determined by
\[
 \sigma_t + \iota(Y_t)\o_t = 0,
\]
although we need to be careful that we can actually integrate $Y_t$ all the way
to time 1.
This can be done because our forms are all cylindrical at infinity, so the
vector fields obtained above will all scale according to $e^r$ as we move in
the $r$-direction along the collar.  This bound is enough to ensure we can
integrate to a flow.
\end{proof}

\subsection{Proof of Theorem \ref{thm:main}}
To prove Theorem \ref{thm:main} we just apply Lemma \ref{thm:moser} in our
case.  Let $\o_1,\o_2$ be the exact forms induced on $X_1,X_2$ respectively and
suppose, for a contradiction, that there exists a diffeomorphism $\phi \colon
X_2 \to X_1$ such that $\phi^*(\o_1)=\o_2$.

Then we also consider the deforming 2-forms $\eta_2$ and $\phi^*(\eta_1)$
defined on
$X_2$ and by rescaling we may assume without loss of generality that these two
2-forms are cohomologous (since $H^2(X_i;\RR) =\RR$).  We now consider the
family of 2-forms on $X_2$
\[
 \Omega_t = (1-t)(\o_2 +\eta_2) + t\phi^*(\o_1 +\eta_1) =\o_2 + t\phi^*(\eta_1)
+ (1-t)\eta_2.
\]
There exists some compact subset $X_2^{in}$ which is an interior for $X_2$ with
respect to $\Omega_0=\o_2 + \eta_2$, and by the compactness of both $X_2^{in}$
and its boundary, we can say that, after perhaps once more rescaling $\eta_1$
and $\eta_2$ if necessary, $\Omega_t$ makes $X_2^{in}$ a symplectic manifold
with convex boundary for all $t$.  However, $\Omega_t$ is not necessarily
cylindrical for all $t$ so we now change our family $\Omega_t$, by replacing
$\Omega_t|_{X_2^{out}}$ with the completion of $\Omega_t|_{X_2^{in}}$ to get a
new family of cohomologous symplectic forms $\tilde{\Omega}_t$ on $X_2^{in}
\cup_{\partial X_2^{in}}[0,\infty) \times \partial X_2$, which are all
cylindrical on the collar.  Therefore, by Lemma \ref{thm:moser},
$(X_2^{in},\tilde{\Omega}_t)$ are all symplectomorphic.

However, we can choose $X_2^{in}$ sufficiently large that it contains the image
$\phi^{-1}(L)$ of the Lagrangian sphere exhibited in Section
\ref{sec:maydanskiy}.  This is a contradiction of Theorem
\ref{thm:maydanskiyseidel}.

\section{Symplectic cohomology vanishes}
\label{sec:sh}

In this section, we digress from the main theme and discuss
symplectic cohomology.  All symplectic manifolds considered in this section
will be exact and we shall work with $\ZZ/2$-coefficients.  As mentioned in the
Introduction, symplectic cohomology is one of the standard invariants used to
examine and distinguish Liouville domains.  We prove that the symplectic
cohomology $SH^*(X_i;\ZZ/2)$ of $X_1$ and $X_2$ both vanish, thereby showing
that this invariant does not suffice to distinguish between the examples of
this paper, and so a different approach such as that of this paper truly is
needed.

We shall not define symplectic cohomology here; an appropriate definition may
be found in \cite{viterbosh1}, for example.  We shall instead refer to two
results from \cite{abouzaidseidelrecombination}.  In the formulation of these
two lemmas, we consider the Liouville domain $E$ to be built from fibre $M$ and
the collection of vanishing cycles $(V_1, \ldots, V_r)$ according to Lemma
\ref{thm:build}.  We denote by $\Delta_i$ the Lefschetz thimble associated to
$V_i$ in the corresponding Lefschetz fibration $\pi \colon E \to \CC$.

\begin{lem}
\label{thm:shhw}
\emph{(\cite[Property 2.3]{abouzaidseidelrecombination})}
For a Liouville domain $E$, constructed from $(M;V_1, \ldots, V_m)$,
$SH^*(E)=0$ if and only if $HW^*(\Delta_i, \Delta_i)=0$ for all $i$.
\end{lem}

\begin{lem}
\label{thm:removeone}
\emph{(\cite[Property 2.5]{abouzaidseidelrecombination})}
Consider a Liouville domain $E$, constructed from $(M;V_1, \ldots, V_m)$ and
let $E'$ be the Liouville domain built from $(M; V_2, \ldots, V_m)$. Let
$\Delta_i,\Delta'_i$ be the Lefschetz thimbles in $E,E'$ respectively.  If
$HW^*(\Delta_1,\Delta_1)=0$ and $HW^*(\Delta'_i,\Delta'_i)=0$
for all $i$, then $HW^*(\Delta_i,\Delta_i)=0$ for all $i$.
\end{lem}

We also note that if $SH^*(E;\ZZ/2)=0$, then $E$ cannot contain any exact
Lagrangian submanifolds \cite{viterbosh1}.

Lemma \ref{thm:shhw} suffices to prove that Maydanskiy's exotic examples
\cite{maydanskiy} have vanishing symplectic cohomology, as do the exact
symplectic manifolds $X^j_n$ considered in Section \ref{sec:many}.  We now
prove that the exotic examples of Maydanskiy-Seidel, as well as their versions
obtained from adding a 2-handle in the way described in Section
\ref{sec:maydanskiyseidel} have vanishing symplectic cohomology.  Take some
exotic example $X_0$ from \cite{maydanskiyseidel}, as in Figure
\ref{fig:maydanskiyseidel}, but without the extra rightmost critical point.

The proof in \cite{maydanskiyseidel}, as outlined in Section
\ref{sec:maydanskiyseidel}, shows that
$HW^*(\Delta_{m+1},\Delta_{m+1})=0$.  We apply Lemma \ref{thm:removeone} in
this setting, and remark that this lemma still holds if we  remove the final
vanishing cycle instead of the first.  If we restrict to the $A_m$
configuration of vanishing cycles $(V_1,
\ldots, V_m)$ in Figure \ref{fig:maydanskiyseidel}, then $X'_0$ is just
isomorphic to the standard ball.  This
means that if we compute $HW^*(\Delta'_i,\Delta'_i)$, we get zero
as all the Floer groups involved in the definition of
$HW^*(\Delta'_i,\Delta'_i)$ will vanish.  This suffices to prove that
$HW^*(\Delta_i,\Delta_i)=0$ for all $i$, and so
$SH^*(X_0)=0$.

We construct the manifold $X_2$ of Section \ref{sec:maydanskiyseidel} by adding
a 2-handle to $M_m$.  However, because this handle is added away from all the
vanishing cycles, we can just view this as a subcritical handle added to $X_0$,
as opposed to a critical one added to $M_m$ since $X_0$ is a product fibration
in the region where the handle is attached.  Cieliebak's result
\cite{cieliebakhandle} says that $SH^*(X_2)=SH^*(X_0)$ is still zero.  In
particular we have

\begin{thm}
$X_2$ and $X_0$ are both empty as exact symplectic manifolds, in the sense of
containing no exact Lagrangian submanifolds.
\end{thm}

\begin{rmk}
It is sometimes possible to define symplectic cohomology with respect to some
nonexact symplectic form.  Ritter \cite{ritterdeformedsh} shows that, if one
performs a nonexact deformation of the exact symplectic form, then this is the
same as computing the symplectic cohomology of the original structure, but with
coefficients in some twisted Novikov bundle: $SH^*(M, d\theta + \eta)= SH^*(M,
d\theta;\underline{\Lambda}_{\tau \eta})$.  This has implications for the
existence of exact Lagrangians and it would be interesting to compare the
results of this paper with this viewpoint.
\end{rmk}

\section{Many inequivalent exotic symplectic forms}
\label{sec:many}

\subsection{An invariant}
We shall now extend the ideas of Section \ref{sec:end} in order to prove
Theorem \ref{thm:many}.  Suppose we have a symplectic manifold $(E,\o)$ which
is convex at infinity and such that the map $H^2(E; \RR) \to H^2(\partial E;
\RR)$ is zero.  Then, given any cohomology class $\eta \in H^2(E; \RR)$, we can
construct a deformation of $E$ in the sense of Section
\ref{sec:maydanskiyseidel} in the direction of $\eta$, in other words
$[\omega_t] = [\o + t\epsilon \eta] \in H^2(E;\RR)$ for some small $\epsilon
>0$.

Suppose in addition that $(E,\o)$ contains no homologically essential
Lagrangian sphere.  We denote by $\Gamma_1(E,\o)$ the set of directions $l \in
\mathbb{P}(H^2(E; \RR))$ such that, after constructing a ``small'' deformation of
$(E,\o)$ in direction $l$, we \emph{still have no} homologically essential
Lagrangian sphere.  The Moser-type argument from Section \ref{sec:end} says
that this set is well-defined (up to projective linear equivalence).

We can likewise consider the invariants $\Gamma_k(E,\o)$, which are the set of
$k$-planes $P_k$ in the Grassmanian $Gr(H^2(E; \RR))$, such that we get no
homologically essential Lagrangian sphere \emph{for every} direction $l$
contained in $P_k$.  These are again invariants up to the correct notion of
linear equivalence, and so in particular, if we get a finite set of such
planes, the cardinality of $\Gamma_k(E,\o)$ is invariant.

\subsection{The construction}
We now extend the construction of Maydanskiy \cite{maydanskiy} to exhibit, for
any $n \geq 1$, a Liouville manifold which admits $n+1$ symplectic forms $\o_k$
all of which have no homologically essential exact Lagrangian sphere (in fact
which have vanishing symplectic cohomology $SH^*(E,\o_k;\ZZ/2)$ and therefore
no exact Lagrangian submanifolds), but such that there exists no diffeomorphism
$\phi$ of $E$ such that $\phi^*\o_i = \o_j$ for $i \neq j$.

Take the points $0,1, \ldots, n+1 \in \CC$ and consider two paths in $\CC$ as
in Figure \ref{fig:many}.  The first $\gamma_0$ joins the extreme crosses and
goes over all the others.  We have some choice in the second path and denote by
$\gamma_j$ the path which goes below the points $1, \ldots, j$ and then over
$j+1, \ldots, n$.  (We include here the possibility that the second path
actually goes over all central crosses and in this case just consider it to be
another copy of $\gamma_0$.)

\begin{figure}[h!]
\[
\xy
(-20,0)*{\times}; (40,0)*{\times}; **\crv{(5,15)};
(-20,0)*{}; (40,0)*{}; **\crv{(-18,15)&(3,-15)};
(-20,-3)*{0};
(40,-3)*{n+1};
(-10,0)*{\times}; (-10,-10)*{}; **\dir{.};
(10,0)*{\cdots \cdots \cdots};
(0,0)*{\times}; (0,-10)*{}; **\dir{.};
(20,0)*{\times}; (20,-10)*{}; **\dir{.};
(30,0)*{\times}; (30,-10)*{}; **\dir{.};
(-6,-6)*{\gamma_j};
(29,6)*{\gamma_0};
\endxy
\]
\caption{}
\label{fig:many}
\end{figure}

With the same conventions as before, having made our choice of $\gamma_j$, we
can associate to Figure \ref{fig:many} the 6-dimensional manifold
$(X^j_n,\o_j)$, which is diffeomorphically $T^*S^3$ with $n$ 2-handles
attached.  It is the total space of a Lefschetz fibration whose generic fibre
is the $A_{n+1}$ Milnor fibre, which we shall denote $M_{n+1}$.  Associated to
each dotted line we get a Lagrangian 2-ball $B_i \subset M_{n+1}$ for $1\leq i
\leq n$, and we denote by $V_0$ and $V_j$ the two matching paths associated to
the paths $\gamma_0$ and $\gamma_j$.  If $\gamma_j=\gamma_0$, the 6-manifold we
obtain clearly contains a Lagrangian $S^3$, coming from the zero-section of
$T^*S^3$.  We shall denote by $\Delta_0,\Delta_j$ the Lefschetz thimbles
associated to the two critical points of the Lefschetz fibration $\pi \colon
X^j_n \to \CC$.

$H_2(M_{n+1} ; \RR) \cong \RR^{n+1}$ and we shall choose as our standard basis
the spheres $A_i$ given by straightline paths joining adjacent critical points
$i-1$ and $i$ in Figure \ref{fig:many}.  When included into our total space,
these all determine nonzero homology classes in $E$, but now with the relation
$\sum{A_i}=0$.  We shall therefore choose to identify $H^2(E;\RR)$ with the
$n$-dimensional vector space $V = \{ v \in \RR^{n+1} : \sum v_i = 0\}$.

Pick some vector $\mathbf{v} = (v_1, \ldots, v_n,v_{n+1}) \in V$.  By the same
process as in Section \ref{sec:maydanskiy}, we can construct a deformation of
the symplectic structure on $M_{n+1}$, by adding on 2-forms in the regions
between the critical point weighted according to the components.  The condition
on $\mathbf{v}$ means the that the homological obstruction to the matching
paths above defining matching cycles vanishes, so we can once more build the
corresponding deformation of $(X^j_n,\o_j)$.  We are interested in what choices
of $j$ and $\mathbf{v}$ mean that $(X^j_n,\o_j)$ contains a Lagrangian sphere
after the deformation coresponding to $\mathbf{v}$.

We first observe that, as in Section \ref{sec:maydanskiy}, we shall get a
Lagrangian sphere in $X^j_n$ when we can ``lift'' $V_j$ over the critical
points and onto $V_0$.  For this to be true, we need
\[
\sum^k_r v_r \neq 0 \mbox{ for all } k \leq j.
\]
In this case we shall get a Lagrangian sphere in $X_j^n$ once we perturb in the
direction of $\mathbf{v}$.  We shall now show that in all other cases we do not
get such a sphere.

Fix some direction $\mathbf{v} \in V$.  In what follows, we shall as before
denote by $HF^*_t$ the Floer cohomology group computed with respect to the
time-$t$ deformation of $\o$ in the direction of $\mathbf{v}$.  For the same
reasons as already discussed, all these groups are well-defined (perhaps after
rescaling $\mathbf{v})$.

Suppose that there is a homologically essential Lagrangian sphere $L \subset
(X_n^j,\o_t)$.  Then, as in Section \ref{sec:maydanskiyseidel}, we must have $L
\cdot \Delta_j \neq 0$, which implies that $HW_t^*(\Delta_j, \Delta_j) \neq 0$.
This wrapped Floer group fits in an exact triangle as before.
\begin{eqnarray}
\label{eqn:triangle}
\xymatrix{
{HF^*_t(\Delta_j,\Delta^{\epsilon}_j)} \ar[r] & {HF^*_t(\Delta_j,\Delta_j^{2\pi
+\epsilon})} \ar[d]  \\
 & {HF^*_t(\mu(V_j),V_j).} \ar[ul]
}
\end{eqnarray}
where the bottom group is calculated in the fibre $E_z$.  Here $\mu$ is, up to
isotopy, $\tau_{V_0} \circ \tau_{V_j}$, so we shall need to consider the group
$HF^*_t(\tau_{V_0}V_j,V_j)$.

The argument in this section largely follows that found in \cite{maydanskiy},
from where we reproduce the following basic observation.

\begin{lem}
If we have an exact triangle of graded vector spaces
\[
\xymatrix{
{K} \ar[r]^F & {L} \ar[d]  \\
 & {M,} \ar^{[1]}[ul]
}
\]
then $\rk(M) = \rk(K) + \rk(L) - 2 \rk(\im(F))$.
\end{lem}

We shall consider this lemma applied to the following triangle coming from the
long exact sequence in \cite{seidelles}.
\[
\xymatrix{
{HF_t^*(V_0,V_j) \otimes HF_t^*(V_j,V_0)} \ar[r] & {HF_t^*(V_j,V_j)} \ar[d]  \\
 & {HF_t^*(\tau_{V_0}V_j, V_j).} \ar[ul]
}
\]
\begin{rmk}
To apply Seidel's long exact sequence in this nonexact setting, we can no
longer filter the Floer cochain groups by the symplectic action, as discussed
in Section \ref{sec:fukaya}.  However, we can introduce a filtration by powers
of our formal Novikov parameter $q$.  This will give us an appropriate
$\ZZ$-filtration as the energy spectrum of the (unperturbed) holomorphic curves
$u$ will form a discrete set.
\end{rmk}

Consider now the Lagrangian balls $B_i$ associated to the dotted paths in
Figure \ref{fig:many} and suppose there is an $i$ such that $HF_t^*(V_j,B_i)$
is nonzero.  Then the product
\[
HF_t^*(V_0,V_j) \otimes HF_t^*(V_j,V_0) \to HF_t^*(V_j,V_j) \cong H^*(S^2)
\]
does not contain the identity in its image, because if it did, then the
composite
\[
HF_t^*(V_j,B_i) \otimes HF_t^*(V_0,V_j) \otimes HF_t^*(V_j,V_0) \to
HF_t^*(V_j,B_i)
\]
would hit the identity despite factoring through $HF_t^*(V_0,B_i)$ which is
zero as these Lagrangians are disjoint.  Here we use the fact that the product
structure on Floer cohomology is associative.  However, the fundamental class
of $H^2(S^2)$ is in the image, by Poincar\'{e} duality for Floer cohomology.

So, when we consider the ranks of the groups in the above triangle, we see that
\begin{lem}
If $HF_t^*(V_j,B_i)\neq 0$ for any $i$, then $\rk HF_t^*(\tau_{V_0} V_j,
V_j)=4$.
\end{lem}

We now consider the triangle \eqref{eqn:triangle} relating the first few terms
in the system of groups computing $HW_t^*(\Delta_j,\Delta_j)$.  Again, by
computing ranks we see that, if $\rk HF_t^*(\tau_{V_0} V_j, V_j)=4$, then the
rank of the image of the horizontal map must be zero, and therefore take 1 to
0, which in turn forces $HW_t^*(\Delta_j,\Delta_j)=0$.  We conclude

\begin{lem}
If $HF_t^*(V_j,B_i) \neq 0$ for any $i$, then there exists no homologically
essential Lagrangian sphere.
\end{lem}

For $i>j$, $V_j$ and $B_i$ are disjoint so $HF^*_t(V_j, B_i)=0$ is automatic.
For $i \leq j$, the criterion that $HF_t^*(V_j,B_i)$ be nonzero corresponds to
\[
\sum^k_r v_r \neq 0 \mbox{ for all } k\leq i
\]
since, in the fibre where the paths defining $V_j$ and $B_i$ intersect we
either get disjoint circles or instead two copies of some circle $C$ whose self-Floer
cohomology $HF^*_t(C,C) \cong H^*(C)$ is nonzero.

\begin{rmk}
In particular, the above argument shows that, in the undeformed case,
$HW^*(\Delta_j,\Delta_j)=0$.  A similar argument also shows that
$HW^*(\Delta_0,\Delta_0)=0$, which, by Lemma \ref{thm:shhw}, proves that, for
our undeformed exact symplectic manifolds $SH^*(X_n^i)=0$ for all $i$.
\end{rmk}

Therefore, if we consider the $(n-1)$-Grassmanian invariant
$\Gamma_{n-1}(X_n^j)$, we see that the planes for which we get no Lagrangians
appearing are, in our choice of basis, precisely those $(n-1)$-planes defined
by any one of the equations
\[
\sum^k_r v_r = 0 \mbox{ for some } k\leq j,
\]
so that $\Gamma_{n-1}(X_n^j)$ is a set consisting of $j$ points.

We now have, for $1 \leq i \leq n$, exact symplectic manifolds such $X^i_n$ is
not symplectomorphic to $X_n^j$ for $i \neq j$, even though neither contains
any exact Lagrangian submanifolds.  Our final manifold $(X^{n+1}_n, \o_{n+1})$
simply comes from adding $n$ handles to some exotic Maydanskiy-Seidel example,
just as in Figure \ref{fig:maydanskiyseidel}.  The same argument as in Section
\ref{sec:maydanskiyseidel} will show that $\Gamma_{n-1}(X^{n+1}_n,
\o_{n+1})=Gr_{n-1}(\RR^n)$, so $X^{n+1}_n$ cannot be symplectomorphic to any of
the $X^i_n$ for $i \leq n$.  This completes the proof of Theorem
\ref{thm:many}.

\end{document}